\documentclass[10pt]{amsart}
\usepackage[utf8]{inputenc} 
\usepackage[T1]{fontenc}
\usepackage{amsfonts,amscd}
\usepackage[a4paper, margin=1.25in]{geometry}
\usepackage[parfill]{parskip}
\usepackage{xcolor}
\usepackage{comment}
\usepackage{amssymb,latexsym,amsmath,amsthm,enumitem,mathrsfs}
\usepackage{fullpage}
\usepackage{fancyhdr}
\usepackage{hyperref}
\usepackage{diagbox}
\usepackage{array}
\usepackage{tikz}
\usetikzlibrary{arrows}
\usetikzlibrary{positioning}
\usepackage{float}\usepackage{mathtools} 
\usepackage{dsfont}
\usepackage[titletoc,title]{appendix}
\newtheorem{theorem}{Theorem}
\newtheorem{defn}[theorem]{Definition}

\newtheorem{lemma}[theorem]{Lemma}

\mathtoolsset{showonlyrefs}

\allowdisplaybreaks
\usepackage{graphicx} 

\title{Moments and sign changes of symmetric power $L$-function coefficients over sums of squares}
\author{Jewel Mahajan}
\address{(Jewel Mahajan) Department of Mathematical Sciences, Indian Institute of Science Education and Research (IISER) Berhampur, Ganjam, Odisha 760003, India}
\email{jewelmahajan421@gmail.com}
\author{Arnab Mitra}
\address{(Arnab Mitra) School of Mathematical Sciences, National Institute of Science Education and Research,  An OCC of Homi Bhabha National Institute, Bhubaneswar, Via: Jatni, Khurda, Odisha-752050, India}
\email{arnab.mitra@niser.ac.in}
\date{June 2026}

\begin{document}

\begin{abstract}
Let $f$ be a normalised Hecke eigenform of even integral weight for the full modular group $\mathrm{SL}(2,\mathbb{Z})$, let $L(s,\mathrm{sym}^{j}f)$ be the $j$th symmetric power $L$-function attached to $f$, and let $\lambda_{\mathrm{sym}^{j}f}(n)$ denote its $n$th Dirichlet coefficient. For each even integer $m$ with $2 \le m \le 12$, we establish upper bounds for the partial sums of $\lambda_{\mathrm{sym}^{j}f}(n)$ and asymptotic formulas for those of $\lambda_{\mathrm{sym}^{j}f}^{2}(n)$ taken over integers represented as a sum of $m$ squares. As an application, we obtain lower bounds for the number of sign changes of $\lambda_{\mathrm{sym}^{j}f}(n)$ along these sums of $m$ squares.
\end{abstract}

\maketitle

\section{Introduction}

Let $k$ be a positive even integer. Let \( H_k \) be the space of all normalised Hecke eigencusp forms of weight $k$ for the full modular group \(\mathrm{SL}(2,\mathbb{Z})\). Let $f \in H_k $ and \( \lambda_{f}(n) \) be the \( n \)-th normalised Fourier coefficient in the Fourier expansion of \( f(z) \) at the cusp \( \infty \), that is,
\[
f(z)=\sum_{n=1}^{\infty}\lambda_{f}(n)\,n^{\frac{k-1}{2}}e^{2\pi i n z},
\qquad \Im(z)>0,
\]
with \(\lambda_{f}(1)=1\).
The \( L \)-function attached to \( f \) (denoted \( L(s,f) \)) is defined as
\[
L(s,f)=\sum_{n=1}^{\infty}\frac{\lambda_{f}(n)}{n^{s}},
\qquad \Re(s)>1.
\]

In 1974, Deligne~\cite{deligne1974} proved that for any prime \( p \), there exist two complex numbers \( \alpha_{f}(p) \) and \( \beta_{f}(p) \) (known as Satake parameters) such that
\begin{align}\label{eq:deligne1}
\alpha_{f}(p)\beta_{f}(p) = 1,  
\end{align}
\begin{align}\label{eq:deligne2}
|\alpha_{f}(p)| = |\beta_{f}(p)| = 1,  
\end{align}
and
\begin{align}\label{eq:deligne3}
\lambda_{f}(p)=\alpha_{f}(p)+\beta_{f}(p). 
\end{align}
Since \( f \) is primitive, the coefficients \( \lambda_f(n) \) are multiplicative and satisfy the Hecke recurrence
\begin{equation}\label{recurrence}
\lambda_f(p^{\,r+1})=\lambda_f(p)\,\lambda_f(p^{\,r})-\lambda_f(p^{\,r-1}),
\qquad r\ge 1.
\end{equation}
Consequently, for \( \Re(s)>1 \), we have the Euler product
\begin{equation}\label{eulerproduct}
L(s,f)=\prod_{p}\; \sum_{r=0}^{\infty}\frac{\lambda_f(p^{\,r})}{p^{\,rs}} .
\end{equation}

Using induction and the identities \( \lambda_f(p)=\alpha_f(p)+\beta_f(p) \) and \( \alpha_f(p)\beta_f(p)=1 \), together with \eqref{recurrence}, one obtains
\begin{equation}\label{eq:closedform}
\lambda_f(p^{\,r}) 
= \frac{\alpha_f(p)^{\,r+1}-\beta_f(p)^{\,r+1}}{\alpha_f(p)-\beta_f(p)}
= \alpha_f(p)^{\,r}+\alpha_f(p)^{\,r-2}+\cdots+\alpha_f(p)^{-r}
\end{equation}
for all integers \( r\ge 1 \). From this follows Deligne's inequality
\begin{equation}\label{Deligne bound}
|\lambda_{f}(n)|\leq d(n),    
\end{equation}
where \( d(n) \) is the divisor function.

Now using the multiplicativity of \( \lambda_{f}(n) \) and formula \eqref{eq:closedform}, the Euler product \eqref{eulerproduct} can be rewritten as
\begin{equation}\label{Lfactorization}
L(s,f)=\prod_{p}\,
\biggl(1-\frac{\alpha_{f}(p)}{p^{s}}\biggr)^{\!-1}
\biggl(1-\frac{\beta_{f}(p)}{p^{s}}\biggr)^{\!-1}
\end{equation}
in \(\Re(s)>1\), where the product runs over all primes \( p \).

The $j$th symmetric power $L$-function associated to each $f\in H_k $ is defined as
\[
L(s,\mathrm{sym}^{j}f):=\sum_{n=1}^{\infty}\frac{\lambda_{\mathrm{sym}^{j}f}(n)}{n^{s}}=\prod_{p}\prod_{m=0}^{j}\left(1-\frac{\alpha_f(p)^{m}\beta_f(p)^{j-m}}{p^s}\right)^{-1}
\]
for $\Re(s)>1$, where $\lambda_{\mathrm{sym}^{j}f}(n)$ is multiplicative. From the multiplicativity, \eqref{eq:deligne1} and \eqref{eq:deligne2}, we have \begin{align}
    \lambda_{\mathrm{sym}^jf}(n) \ll n^{\epsilon}.
\end{align}

Several authors have studied the average behaviour of these Fourier coefficients. In 2004, Fomenko~\cite{Fomenko2004} was able to prove some results for the symmetric square $L$-functions. He showed that
\[
\sum_{n\leq x}\lambda_{\mathrm{sym}^{2}f}(n)\ll x^{\frac{1}{2}}\log^{2}x.
\]



In 2013, Zhai~\cite{Zhai} proved an asymptotic formula for the $l$th
power sum 
\[
\sum_{\substack{a^{2}+b^{2}\leq x \\ (a,b)\in\mathbb{Z}^{2}}}\lambda_{f}^{l}(a^{2}+b^{2})
\]
for $l \in [2,8]\cap \mathbb{Z}$ and $x\geq 1$. For results related to mean square and higher moments of the coefficients of symmetric square $L$-functions on a certain sequence of positive integers, see~\cite{SharmaSankaranarayanan2023B, SharmaSankaranarayanan2022}.

In~\cite{Sharma2022}, Sharma and Sankaranarayanan established the formulas for the Dirichlet coefficients of the symmetric square $L$-functions over sums of four squares, that is,
\begin{align*}
\sum_{\substack{a_{1}^{2}+\cdots+a_{4}^{2}\leq x \\ (a_{1},\ldots,a_{4})\in\mathbb{Z}^{4}}}\lambda^{2}_{\mathrm{sym}^{2}f}(a_{1}^{2}+\cdots+a_{4}^{2}) = cx^2+O(x^{\frac{9}{5}+\epsilon})
\end{align*}
for any $\epsilon > 0$, where $c $ is an effective constant (see~\cite[Theorem~1.1]{Sharma2022}).

In~\cite{SharmaSankaranarayanan2023}, Sharma and Sankaranarayanan proved that 
\begin{align*}
\sum_{\substack{a_{1}^{2}+\cdots+a_{6}^{2}\leq x \\ (a_{1},\ldots,a_{6})\in\mathbb{Z}^{6}}}\lambda^{2}_{\mathrm{sym}^{j}f}(a_{1}^{2}+\cdots+a_{6}^{2}) = c(j)x^3+O(x^{3-\frac{6}{3(j+1)^2+1}+\epsilon}),    
\end{align*}
where $$c(j) = \frac{16}{3}L(3,\chi_4)\prod_{n=1}^{j}L(1,\mathrm{sym}^{2n}f)L(3,\mathrm{sym}^{2n}f \otimes \chi_4)H_j(3),$$ and $\chi_4$ is the non-principal Dirichlet character modulo $4$. 
Here \[L(s,\mathrm{sym}^{2n}f\otimes\chi_{4}) = \sum_{m=1}^{\infty} \frac{\lambda_{\mathrm{sym}^{2n}f}(m)\chi_4(m)}{m^s}\] and $H_j(s)$ is a Dirichlet series, which is absolutely convergent at $s=3$. The explicit expression for $H_j(s)$  is mentioned in~\cite{SharmaSankaranarayanan2023}.

The best known result for the estimate for the sum of six squares is proved by Liu and Yang~\cite{Liu2024}

\begin{align*}
\sum_{\substack{a_{1}^{2}+\cdots+a_{6}^{2}\leq x \\ (a_{1},\ldots,a_{6})\in\mathbb{Z}^{6}}}\lambda^{2}_{\mathrm{sym}^{j}f}(a_{1}^{2}+\cdots+a_{6}^{2}) = c(j)x^3+O(x^{3-\frac{210}{105(j+1)^2-103}+\epsilon}),  
\end{align*}

where $c(j)$ is a nonzero constant.

In~\cite{Junjie2026}, Wang and Wang proved that

\begin{align}
    \sum_{\substack{a_{1}^{2}+\cdots+a_{8}^{2}\leq x \\ (a_{1},\ldots,a_{8})\in\mathbb{Z}^{8}}}\lambda_{\mathrm{sym}^{j}f}(a_{1}^{2}+\cdots+a_{8}^{2}) = \begin{cases}
        O(x^{4-\frac{2}{2j+3}+\epsilon}) & \text{ if } j \geq 3 \\
        O(x^{\frac{25}{7}+\epsilon}) & \text{ if } j=2
    \end{cases}   
\end{align}

\begin{align*}
\sum_{\substack{a_{1}^{2}+\cdots+a_{8}^{2}\leq x \\ (a_{1},\ldots,a_{8})\in\mathbb{Z}^{8}}}\lambda^{2}_{\mathrm{sym}^{j}f}(a_{1}^{2}+\cdots+a_{8}^{2}) = c(j)x^4+\begin{cases}
    O(x^{4-\frac{168}{84(j+1)^2-95}+\epsilon}) & \text{ if } j \geq 3 \\
    O(x^{\frac{654}{174}+\epsilon}) & \text{ if } j =2
\end{cases}    
\end{align*}
where $$c(j) = 4\zeta(4)\prod_{n=1}^{j}L(4, \mathrm{sym}^{2n}f)^{16}L(1, \mathrm{sym}^{2n}f)^{16}H_j(4),$$
where $H_j(4)$ is a Dirichlet series. 

On the other hand, the study of sign changes in Fourier coefficients is another interesting topic in analytic number theory.  

In 2013, Meher, Shankhadhar and Viswanadham~\cite{Jaban2013} examined the sign changes of $\{a(n^j)\}$ for $ j\in\{ 2,3,4\}$, where $a(n)$ is the $n$th Fourier coefficient of normalised Hecke eigencusp forms for the full modular group. In 2014, Meher and Murty examined the sign changes of Fourier coefficients of half-integral weight cusp forms in~\cite{Jaban2014_1} and Meher and Tanabe examined the sign changes of Fourier coefficients of Hilbert modular forms in~\cite{Jaban2014}.  In 2025, Kaur and Saha~\cite{Kaur2025} examined the number of sign changes of Fourier coefficients of $SL_2(\mathbb{Z})$ Hecke–Maass forms at sum of two squares. 

In 2026, Kaur~\cite{Kaur2026} proved that for any $j \geq 2$, $x$ sufficiently large, the number of sign changes in $\lambda_{\mathrm{sym}^jf}(n)$ over the sums of two squares and between $x$ and $2x$, is at least $x^{1-\delta_j}$, where $$\frac{21j^2+42j+19}{21j^2+42j+40}< \delta_j < 1.$$

We note that the number of representations of a natural number as a sum of $k$ squares, where $k$ varies in $[2,12] \cap 2\mathbb{Z}$, can be written as a linear combination of divisor functions, twisted with the non-principal Dirichlet character modulo $4$. Using this, we examine an asymptotic formula for the average behaviour of $\lambda_{\mathrm{sym}^jf}(n)$ and $\lambda_{\mathrm{sym}^jf}^2(n)$, over sum of $k$ squares, where $k \in [2,12] \cap 2\mathbb{Z}$. In particular, we prove the following results. 
\begin{theorem} \label{thm:main(0)}
    Let $f \in H_k$, $j \geq 2$ be a fixed integer. Then for any $\epsilon > 0$, we have $$\sum_{n \leq x}\lambda_{\mathrm{sym}^jf}(n) = O\left( x^{1-\frac{2}{j+3}+\epsilon} \right).$$
\end{theorem}

\begin{theorem} \label{thm:main}
    Let $f \in H_k$ and $j \geq 2$ be a fixed integer. Then for any $\epsilon > 0$, we have $$\sum_{\substack{a_1^2+a_2^2 \leq x \\ (a_1,a_2) \in \mathbb{Z}^2}}\lambda_{\mathrm{sym}^jf}(a_{1}^{2}+a_2^2) = O \left( x^{1-\frac{1}{j+1}+\epsilon}\right).$$
\end{theorem}

\begin{theorem} \label{thm:main(468)}
    Let $f \in H_k$, $j \geq 2$ be a fixed integer and $ m \in \{4,6,8\}$. Then for any $\epsilon > 0$, we have $$\sum_{\substack{a_{1}^{2}+\cdots+a_{m}^{2} \leq x \\ (a_{1},\ldots,a_{m})\in\mathbb{Z}^{m}}}\lambda_{\mathrm{sym}^jf}(a_{1}^{2}+\cdots+a_{m}^{2}) = O\left( x^{\frac{m}{2}-\frac{2}{j+3}+\epsilon} \right).$$
\end{theorem}

\begin{theorem} \label{thm:main1(0)}
        Let $f \in H_k$, $j \geq 2$ be a fixed integer. Then for any $\epsilon > 0$, we have 
         $$\sum_{n \leq x}\lambda_{\mathrm{sym}^jf}^2(n) = c_{0,j,f}x + O\left( x^{1-\frac{2}{(j+1)^2}+\epsilon} \right),$$ where $c_{0,j,f}$ is a constant that depends on $j$, $f$.
\end{theorem}

\begin{theorem} \label{thm:main1}
    Let $f \in H_k$ and $j \geq 2$ be a fixed integer. Then for any $\epsilon > 0$, we have $$\sum_{\substack{a_1^2+a_2^2 \leq x \\ (a_1,a_2) \in \mathbb{Z}^2}}\lambda_{\mathrm{sym}^jf}^2(a_{1}^{2}+a_2^2) =c_{1,j,f}x+ O \left(x^{1-\frac{1}{(j+1)^2}+\epsilon}\right),$$  where $c_{1,j,f}$ is a constant that depends on $j$ and $f$.  
\end{theorem}

\begin{theorem} \label{thm:main1(468)}
        Let $f \in H_k$, $j \geq 2$ be a fixed integer and $ m \in \{4,6,8\}$. Then for any $\epsilon > 0$, we have 
         $$\sum_{\substack{a_{1}^{2}+\cdots+a_{m}^{2} \leq x \\ (a_{1},\ldots,a_{m})\in\mathbb{Z}^{m}}}\lambda_{\mathrm{sym}^jf}^2(a_{1}^{2}+\cdots+a_{m}^{2}) = c_{m,j,f}x^\frac{m}{2} + O\left( x^{\frac{m}{2}-\frac{2}{(j+1)^2}+\epsilon} \right),$$ where $c_{m,j,f}$ is a constant that depends on $j$, $f$ and $m$.
\end{theorem}

In the sum of $10$ and $12$ squares, there exists another term besides the terms involving divisor sums and Dirichlet character, as we can see in \eqref{eq:r_10&l_5&v_5} and \eqref{eq:r_12&l_6}. But we will have the similar formula for the partial sum of $\lambda_{\mathrm{sym}^jf}(n)$ and $\lambda_{\mathrm{sym}^jf}^2(n)$ over the sum of $10$ and $12$ squares. More precisely, 

\begin{theorem} \label{thm:main(10,12)}
Let $j\geq 2$ be a fixed integer. For any fixed $\epsilon>0$ and all sufficiently large $x$, we have
\begin{align}
\sum_{\substack{a_{1}^{2}+\cdots+a_{m}^{2}\leq x \\ (a_{1},\ldots,a_{m})\in\mathbb{Z}^{m}}}\lambda_{\mathrm{sym}^{j}f}(a_{1}^{2}+\cdots+a_{m}^{2})=O\left(x^{\frac{m}{2}-\frac{2}{j+3}+\epsilon}\right),
\end{align}
for $m = 10,12$.
\end{theorem}

\begin{theorem} \label{thm:main1(10,12)}
Let $j\geq 2$ be a fixed integer. For any fixed $\epsilon>0$ and all sufficiently large $x$, we have
\[
\sum_{\substack{a_{1}^{2}+\cdots+a_{m}^{2}\leq x \\ (a_{1},\ldots,a_{m})\in\mathbb{Z}^{m}}}\lambda_{\mathrm{sym}^{j}f}^2(a_{1}^{2}+\cdots+a_{m}^{2})= D_{j,f,m}x^\frac{m}{2}+O\left(x^{\frac{m}{2}-\frac{2}{(j+1)^2}+\epsilon}\right),
\]
where $m=10,12$ and $D_{j,f,m}$ is an effective constant depending on $j,f$ and $m$. 
\end{theorem}

In the next theorem, we slightly improve the result for the sign changes of $\lambda_{\mathrm{sym}^jf}(n)$ in~\cite{Kaur2026}, when $n$ varies over a sum of two squares and then we also examine the number of sign changes of  $\lambda_{\mathrm{sym}^jf}(n)$ when $n$ varies over the sum of $4,6,8,10$ and $12$ squares, respectively.

\begin{theorem} \label{thm:mainsign}
    Let $f \in H_k$ and $j \geq 2$ be a fixed integer. Then, for sufficiently large $x$, the sequence $$\{ \lambda_{\mathrm{sym}^jf}(n) | n = a_1^2+a_2^2 , a_i \in \mathbb{Z} \}$$ has at least $x^{1-\delta_j}$ sign changes between $x$ and $2x$, for any $\delta_j $ with $1-\frac{1}{(j+1)^2}  < \delta_j < 1$.
\end{theorem}

\begin{theorem} \label{thm:main1(468)sign}
        Let $f \in H_k$, $j \geq 2$ be a fixed integer and $ m \in \{4,6,8, 10, 12\}$. Then, for sufficiently large $x$, the sequence
        $$\{ \lambda_{\mathrm{sym}^jf}(n) | n = \sum_{i=1}^{m}a_i^2 , a_i \in \mathbb{Z} \}$$ has at least $x^{1-\delta_j}$ sign changes between $x$ and $2x$, for any $\delta_j $ with $1-\frac{2}{(j+1)^2}  < \delta_j < 1-\frac{1}{(j+1)^2}$.
\end{theorem}

\subsection*{Organisation of the article}
This article is organised as follows. In sections \ref{sec:Prel} and \ref{sec:Imp-Lemma}, we discuss preliminaries, important lemmas, and bounds needed to prove our results. Then the proofs of the main theorems concerning the partial sums of the symmetric power $L$-functions attached to Hecke eigenforms are given in the sections \ref{sec:main0}, \ref{sec:main}, \ref{sec:main(468)}, \ref{sec:main1(0)}, \ref{sec:main1(468)}, \ref{sec:main(10,12)} and \ref{sec:main1(10,12)}. Finally, the sections \ref{sec:mainsign} and \ref{sec:main1(468)sign} deal with the number of sign changes of $\lambda_{\mathrm{sym}^jf}(n)$ over the sum of squares.

\section{Preliminaries} \label{sec:Prel}

Let
\[
r_k(n) := \#\bigl\{(n_1, n_2, \dots, n_k) \in \mathbb{Z}^k : n_1^2 + n_2^2 + \cdots + n_k^2 = n\bigr\},
\]
where we count all ordered $k$-tuples of integers $(n_1,\dots,n_k)$ satisfying the equation, 
including zeros and treating different signs and permutations as distinct. 

We now define the functions $r_m(n)$, where $m=2,4,6,8,10,12$, which are defined as follows
\begin{defn} \cite[p.~121]{Grosswald}
For any positive integer $n$, define
\begin{align} 
    r_2(n) &= 4\sum_{d \mid n}\chi_4(d), \label{eq:r_2(n)} \\
    r_4(n) &= 8\sum_{d \mid n}d, \label{eq:r_4(n)} \\
    r_6(n) &= 16\sum_{d \mid n}d^2\chi_4\left(\frac{n}{d}\right) - 4\sum_{d \mid n}d^2\chi_4(d), \label{eq:r_6(n)} \\
    r_8(n) &= 16\sum_{d \mid n}(-1)^{n+d}d^3, \label{eq:r_8(n)} \\
    r_{10}(n) &= \frac{64}{5}\left\{ \sum_{d \mid n}\chi(d')d^4 + \frac{1}{16}\sum_{d \mid n}\chi(d)d^4 \right\} + \frac{32}{5}a_n, \label{eq:r_10} \\
    r_{12}(n) &= 8\sum_{d \mid n}(-1)^{n+d+\frac{n}{d}-1}d^5 + 16b_n, \label{eq:r_12}
\end{align}
where $\chi_4$ is the non-principal Dirichlet character modulo $4$, that is,
\begin{equation}
    \chi_4(m) = 
    \begin{cases} 
         1, & \text{if } m \equiv 1 \pmod 4, \\
        -1, & \text{if } m \equiv 3 \pmod 4, \\
         0, & \text{if } m \text{ is even.}
    \end{cases}
\end{equation}
Here $a_n$ is defined via the identity
\[
\theta_2^4 \, \theta_3^2 \, \theta_4^4 = 16 \sum_{n=1}^\infty a_n q^n \qquad (q = e^{2\pi i z}),
\]
where the classical theta functions are given by
$$\theta_2= 2q^{\frac{1}{4}}\prod_{m=1}^{\infty}(1-q^{2m})(1+q^{2m})^2,$$
$$\theta_3= \prod_{m=1}^{\infty}(1-q^{2m})(1+q^{2m-1})^2,\text{ and}$$
$$\theta_4= \prod_{m=1}^{\infty}(1-q^{2m})(1-q^{2m-1})^2 \quad (|q|<1).$$ 
Similarly, $b_n$ is defined via the identity
\[
\left (\frac{\theta_1'}{\pi}\right)^4 = 16 \sum_{n=1}^\infty b_n q^n ,
\]
where 
$$ \theta_1' = 2\pi q^{1/4}\prod_{m=1}^{\infty}(1-q^{2m})^3 \quad (|q| < 1).$$
\end{defn}
\begin{defn} We define the arithmetic functions as follows.
    \begin{align} 
        l_1(n) &= \sum_{d \mid n}\chi_4(d), \label{eq:l_1(n)} \\
        l_2(n) &= \sum_{d \mid n}d, \label{eq:l_2(n)} \\
        l_3(n) &= \sum_{d \mid n}d^2\chi_4\left( \frac{n}{d} \right), \quad v_3(n) = \sum_{d \mid n} d^2 \chi_4(d), \label{eq:l_3(n)&v_3(n)} \\
        l_4(n) &= \sum_{d \mid n}(-1)^{n+d}d^3, \label{eq:l_4(n)} \\
        l_5(n) &= \sum_{d \mid n} \chi\!\left(\frac{n}{d}\right) d^4, \qquad v_5(n) = \sum_{d \mid n} \chi(d)d^4, \label{eq:l_5(n)&v_5(n)} \\
        l_6(n) &= \sum_{d \mid n}(-1)^{n+d+\frac{n}{d}-1}d^5. \label{eq:l_6(n)}
    \end{align}
\end{defn}

Observing the definitions of the arithmetic functions, we have
\begin{align} 
    r_2(n) &= 4l_1(n) \ll n^\epsilon, \label{eq:r_2&l_1} \\
    r_4(n) &= 8l_2(n) \ll n^{1+\epsilon}, \label{eq:r_4&l_2} \\
    r_6(n) &= 16l_3(n) - 4v_3(n) \ll n^{2+\epsilon}, \label{eq:r_6&l_3&v_3} \\
    r_8(n) &= 16l_4(n) \ll n^{3+\epsilon}, \label{eq:r_8&l_4} \\
    r_{10}(n) &= \frac{64}{5}l_5(n) + \frac{4}{5}v_5(n) + \frac{32}{5}a_n, \label{eq:r_10&l_5&v_5} \\
    r_{12}(n) &= 8l_6(n) + 16b_n, \label{eq:r_12&l_6}
\end{align}
where $\epsilon > 0$, and we note that in general, $r_{m}(n) \ll n^{\frac{m}{2}-1+\epsilon}$ for $m=2, 4, 6, 8, 10, 12$.

We have 
\begin{align}
    &\sum_{\substack{a_{1}^{2}+\cdots+a_{m}^{2}\leq x \\ (a_{1},\ldots,a_{m})\in\mathbb{Z}^{m}}}\lambda_{\mathrm{sym}^{j}f}\left(\sum_{i=1}^{m}a_{i}^{2}\right) \notag \\
    &\quad = \sum_{n\leq x}\lambda_{\mathrm{sym}^{j}f}(n)\sum_{\substack{n=a_{1}^{2}+\cdots+a_{m}^{2} \\ (a_{1},\ldots,a_{m})\in\mathbb{Z}^{m}}} 1 \\
    &\quad = \sum_{n\leq x}\lambda_{\mathrm{sym}^{j}f}(n)r_{m}(n). \notag
\end{align}

Now, using \eqref{eq:r_2&l_1}, \eqref{eq:r_4&l_2}, \eqref{eq:r_6&l_3&v_3},  \eqref{eq:r_8&l_4}, \eqref{eq:r_10&l_5&v_5}, \eqref{eq:r_12&l_6} and the equation above, we have
\begin{align} 
    \sum_{n\leq x}\lambda_{\mathrm{sym}^{j}f}(n)r_{2}(n) &= 4\sum_{n\leq x}\lambda_{\mathrm{sym}^{j}f}(n)l_1(n), \label{eq:Lmbda2Def} \\
    \sum_{n\leq x}\lambda_{\mathrm{sym}^{j}f}(n)r_{4}(n) &= 8\sum_{n\leq x}\lambda_{\mathrm{sym}^{j}f}(n)l_2(n), \label{eq:Lmbda4Def} \\
    \sum_{n\leq x}\lambda_{\mathrm{sym}^{j}f}(n)r_{6}(n) &= 16\sum_{n\leq x}\lambda_{\mathrm{sym}^{j}f}(n)l_3(n) - 4\sum_{n\leq x}\lambda_{\mathrm{sym}^{j}f}(n)v_3(n), \label{eq:Lmbda6Def} \\
    \sum_{n\leq x}\lambda_{\mathrm{sym}^{j}f}(n)\chi(n)r_{8}(n) &= 16\sum_{n\leq x}\lambda_{\mathrm{sym}^{j}f}(n)\chi(n)l_4(n), \label{eq:Lmbda8Def} \\
    \sum_{n\leq x}\lambda_{\mathrm{sym}^{j}f}(n)r_{10}(n) &= \frac{64}{5}\sum_{n\leq x}\lambda_{\mathrm{sym}^{j}f}(n)l_5(n) + \frac{4}{5}\sum_{n\leq x}\lambda_{\mathrm{sym}^{j}f}(n)v_5(n) + \frac{32}{5}\sum_{n\leq x}\lambda_{\mathrm{sym}^{j}f}(n)a_n, \label{eq:Lmbda10Def} \\
    \sum_{n\leq x}\lambda_{\mathrm{sym}^{j}f}(n)r_{12}(n) &= 8\sum_{n\leq x}\lambda_{\mathrm{sym}^{j}f}(n)l_{6}(n) + 16\sum_{n\leq x}\lambda_{\mathrm{sym}^{j}f}(n)b_n. \label{eq:Lmbda12Def}
\end{align}

\section{Important Lemmas and Bounds} \label{sec:Imp-Lemma}

Note that \eqref{eq:deligne2} yields  $\left|1-\frac{\alpha^{j-1}(p)\beta^i(p)}{p^s}\right| \geq 1-\frac{1}{p^\sigma}>0$  for $\Re (s)=\sigma>1$. Therefore,
\begin{align*}
        |L(s,\mathrm{sym}^jf)| &\leq \prod_{p}\prod_{i=0}^{j}\left(1-\frac{1}{p^\sigma} \right)^{-1} 
        =\prod_{i=0}^{j}\zeta(\sigma) = \zeta(\sigma)^{j+1}
        = \sum_{n=1}^{\infty}\frac{d_{j+1}(n)}{n^\sigma},
\end{align*} 
where $d_{j+1}(n)$ is the number of ways of expressing $n$ as a product of $j+1$ factors. 
Since $d_{k}(n) \le d(n)^{k-1}$ for positive integers $k$ and $n$, and since $d(n) \ll_{\epsilon} n^{\epsilon}$ for any $\epsilon > 0$, we obtain
\[
d_{k}(n) \ll_{k,\epsilon} n^{\epsilon} \quad \text{for any } \epsilon > 0.
\]
Therefore, the Dirichlet series for $L(s, \operatorname{sym}^j f)$ is absolutely convergent for $\Re(s) > 1$.

Note that \eqref{eq:deligne1} and \eqref{eq:deligne2} imply
\[
|\lambda_{\operatorname{sym}^{j}f}(n)| \le d_{j+1}(n).
\]
Consequently, for any $\epsilon>0$,
\begin{equation} \label{eq:lambdaSymbound}
|\lambda_{\operatorname{sym}^{j}f}(n)| \ll_{j,\epsilon} n^{\epsilon}.
\end{equation}

Since $\lambda_{\operatorname{sym}^{j}f}(n)$ is multiplicative, $L(s,\operatorname{sym}^{j}f)$ admits an Euler product
\begin{equation} \label{eq:eulerproduct}
L(s,\operatorname{sym}^{j}f) = \prod_{p} \Bigl( 1 + \frac{\lambda_{\operatorname{sym}^{j}f}(p)}{p^{s}} 
    + \frac{\lambda_{\operatorname{sym}^{j}f}(p^{2})}{p^{2s}} + \dots \Bigr),
\end{equation}
which is absolutely convergent for $\Re(s) > 1$.

Observe that
\begin{align} \label{eq:lambdaatp}
\lambda_{\mathrm{sym}^{j}f}(p)=\sum_{m=0}^{j}\alpha^{j-m}(p)\beta^{m}(p).
\end{align}

Moreover, Hecke theory gives the relation $\lambda_{\operatorname{sym}^{j}f}(p) = \lambda_f(p^{j}) $ for each prime $p$.
\begin{lemma}\label{lambdasquare}
Let $f$ be a Hecke eigenform with Satake parameters $\alpha_p,\beta_p$ satisfying $\alpha_p\beta_p=1$.
For any integer $j \ge 1$ and any prime $p$,
\[
\lambda_{\mathrm{sym}^{j}f}^{2}(p) \;=\; 1 \;+\; \sum_{\ell=1}^{j} \lambda_{\mathrm{sym}^{2\ell}f}(p).\]
\end{lemma}
\begin{proof}
Write $\lambda := \lambda_{\mathrm{sym}^{j}f}(p) = \sum_{m=0}^{j} \alpha_p^{\,j-m}\beta_p^{\,m}$.
Then, using $\beta_p = \alpha_p^{-1}$, we obtain
\[
\lambda^{2} = \sum_{m=0}^{j}\sum_{m'=0}^{j} \alpha_p^{\,2j-(m+m')}\beta_p^{\,m+m'}
      = \sum_{t=0}^{2j} N_t \, \alpha_p^{\,2j-t}\beta_p^{\,t}= \sum_{t=0}^{2j} N_t \, \alpha_p^{\,2j-2t},
\]
where $N_t$ counts pairs $(m,m')$ with $m+m'=t$ and $0\le m,m'\le j$.
One has $N_t = t+1$ for $0\le t\le j$ and $N_t = 2j-t+1$ for $j<t\le 2j$. In particular, $N_t = N_{2j-t}$ for $0\le t\le 2j$. Therefore, 
\begin{align}
    \lambda^{2} = &\sum_{t=0}^{2j} N_t \, \alpha_p^{\,2j-2t}\\
      =&N_{j}+ \sum_{t=0}^{j-1} N_t \, \alpha_p^{\,2j-2t}+ \sum_{t=j+1}^{2j} N_{2j-t} \, \alpha_p^{\,2j-2t}\quad (\text{since } N_t = N_{2j-t})\\
      =&N_{j}+ \sum_{t=0}^{j-1} N_t \, \alpha_p^{\,2j-2t}+ \sum_{t=0}^{j-1} N_{t} \, \alpha_p^{\,-2j+2t}\\
      =&N_{j}+ \sum_{k=1}^{j} N_{j-k} \, \left(\alpha_p^{\,2k}+\alpha_p^{\,-2k}\right)\\
      =& (j+1)+ \sum_{k=1}^{j} (j-k+1) \, \left(\alpha_p^{\,2k}+\alpha_p^{\,-2k}\right).
\end{align}
On the other hand, we have
\[
\sum_{l=1}^{j}\lambda_{\mathrm{sym}^{2\ell}f}(p) = \sum_{l=1}^{j}\sum_{u=0}^{2\ell} \alpha_p^{\,2\ell-u}\beta_p^{\,u}= \sum_{l=1}^{j}\sum_{u=0}^{2\ell} \alpha_p^{\,2\ell-2u} = \sum_{l=1}^{j}\sum_{m=-\ell}^{\ell} \alpha_p^{\,2m}.
\]
In the double sum, for \( k \geq 1 \), the term \( \alpha_p^{\,2k} \) (and similarly \( \alpha_p^{\,-2k} \)) occurs once for each \( \ell \) satisfying \( \ell \ge k \), that is, for \( \ell = k, k+1, \dots, j \). Thus, the coefficient of \( \alpha_p^{\,2k} \) or \( \alpha_p^{\,-2k} \) is \( j-k+1 \).
The constant term $\alpha_p^{0}=1$ occurs in every $\lambda_{\operatorname{sym}^{2\ell} f}(p)$ for $\ell = 1, \dots, j$, giving total multiplicity $j$. Therefore,
\[
\sum_{l=1}^{j}\lambda_{\mathrm{sym}^{2\ell}f}(p) = \sum_{l=1}^{j}\sum_{m=-\ell}^{\ell} \alpha_p^{\,2m}=j+ \sum_{k=1}^{j} (j-k+1) \, \left(\alpha_p^{\,2k}+\alpha_p^{\,-2k}\right).
\]This completes the proof.
\end{proof}

We define
\begin{align} \label{lem:F0}
F_{j}^{(0)}(s)=\sum_{n=1}^{\infty}\frac{\lambda_{\operatorname{sym}^{j}f}(n)\,}{n^{s}} = L(s, \mathrm{sym}^jf), \qquad \Re(s)>1.    
\end{align}

\begin{lemma} \label{lem:F1}
Let \(f\) be a normalised primitive holomorphic cusp form of weight \(k\) for \(\operatorname{SL}(2,\mathbb{Z})\), 
and let \(\lambda_{\operatorname{sym}^{j}f}(n)\) denote the \(n\)th normalised Fourier coefficient 
of the \(j\)th symmetric power \(L\)-function attached to \(f\). Define
\[
F_{j}^{(1)}(s)=\sum_{n=1}^{\infty}\frac{\lambda_{\operatorname{sym}^{j}f}(n)\,l_1(n)}{n^{s}}, \qquad \Re(s)>1,
\]
where \(l_1(n)\) is given by \eqref{eq:l_1(n)}. Then \(F_{j}^{(1)}(s)\) admits a factorisation
\[
F_{j}^{(1)}(s)=G_{j}^{(1)}(s)\,H_{j}^{(1)}(s),
\]
in which
\[
G_{j}^{(1)}(s):=L\!\bigl(s,\operatorname{sym}^{j}f\bigr)\; 
        L\!\bigl(s,\operatorname{sym}^{j}f\otimes\chi_4\bigr),
\]
where \(\chi_4\) is the unique non-principal Dirichlet character modulo \(4\), 
and \(H_{j}^{(1)}(s)\) is a Dirichlet series converging absolutely and uniformly 
in the half-plane \(\Re(s)>\frac{1}{2}\).
\end{lemma}

\begin{proof}
Here we follow the steps as in~\cite{SharmaSankaranarayanan2023}. We know that
\[
\lambda_{\operatorname{sym}^{j}f}(n)\,l_1(n) \ll n^{\epsilon} \qquad (\epsilon > 0),
\]
which implies that the Dirichlet series \(F_j^{(1)}(s)\) converges absolutely for \(\Re(s) > 1\). Since \(\lambda_{\operatorname{sym}^{j}f}(n)\) is multiplicative, \(F_j^{(1)}(s)\) therefore admits an Euler product in this half-plane in \(\Re(s) > 1\):
\begin{align} \label{eq:FjEuler1}
    F_{j}^{(1)}(s)=\prod_{p}
        \Bigl(1+
        \frac{\lambda_{\operatorname{sym}^{j}f}(p)\,l_1(p)}{p^{s}}+
        \frac{\lambda_{\operatorname{sym}^{j}f}(p^{2})\,l_1(p^{2})}{p^{2s}}+
        \cdots
        +\frac{\lambda_{\operatorname{sym}^{j}f}(p^{m})\,l_(p^{m})}{p^{ms}}+
        \cdots\Bigr).
\end{align}

Now define the multiplicative function \(b_1(n)\) via its Euler product
\[
\sum_{n=1}^{\infty}\frac{b_1(n)}{n^{s}} 
    := L(s,\operatorname{sym}^{j}f)\;L(s,\operatorname{sym}^{j}f\otimes\chi_4)
    (=: G_{j}(s)),
\]
Therefore, for a prime \(p\), we have $b_1(p) = \lambda_{\operatorname{sym}^{j}f}(p)\, 
       + \lambda_{\operatorname{sym}^{j}f}(p)\,\chi_4.$ Since \(l_1(p) = 1+\chi_4(p)\), we obtain that $b_1(p) = \lambda_{\operatorname{sym}^{j}f}(p)\,l_1(p)$, establishing the desired equality at each prime.

Note that $b_1(p^k) \neq \lambda_{\mathrm{sym}^jf}(p^k)l(p^k) \text{ for all } k>1$ and 
        \begin{align*}
            |b_1(n)| 
            =|(\lambda_{\mathrm{sym}^jf}*\lambda_{\mathrm{sym}^jf}\chi_{4}(n)|
            &\leq \sum_{d|n}|\lambda_{\mathrm{sym}^jf}(d)||\lambda_{\mathrm{sym}^jf}\left(\frac{n}{d}\right)\chi_{4}\left(\frac{n}{d}\right)| \\
            & \leq \sum_{d|n}d^{\epsilon}\left(\frac{n}{d}\right)^\epsilon \leq n^{\epsilon}d(n) \ll_\epsilon n^{\epsilon} \text{ for any } \epsilon >0.
        \end{align*}
        So $\displaystyle\sum_{n=1}^{\infty}\frac{b_1(n)}{n^s}$ is absolutely convergent by $\Re(s) > 1$ and the Euler product ensures that $$\sum_{n=1}^{\infty}\frac{b_1(n)}{n^s} = \prod_p \left(1+\sum_{m \geq1}\frac{b_1(p^m)}{p^{ms}}\right)\quad (\Re (s)>1).$$
        Now, \begin{align*}
            \bigg|\sum_{m=1}^{\infty}\frac{b_1(p^m)}{p^{ms}}\bigg| &\leq \sum_{m=1}^{\infty}\frac{p^{\epsilon m}}{p^{m\sigma}} \leq \sum_{m=1}^{\infty}\frac{p^{\epsilon m}}{p^{(1+2\epsilon)m}} = \sum_{m=1}^{\infty} \frac{1}{p^{m(1+\epsilon)}} = \frac{1}{p^{1+\epsilon}-1} <1 
        \end{align*}
        for $\Re(s) > 1+2\epsilon$.
        
       Let 
\begin{align*}
    A &= \sum_{m=1}^{\infty} \frac{\lambda_{\mathrm{sym}^jf}(p^m)l_1(p^m)}{p^{ms}}, \quad \text{and} \\
    B &= \sum_{m=1}^{\infty}\frac{b_1(p^m)}{p^{ms}} \quad (|B| < 1).
\end{align*}
Therefore,
\begin{align*}
    \frac{1+A}{1+B} 
    &= (1+A)(1 - B + B^2 - \cdots) \\
    &= 1 + A - B - AB + \cdots \\
    &= 1 + \frac{\lambda_{\mathrm{sym}^jf}(p^2)l_1(p^2)-b_1(p^2)}{p^{2s}} + \cdots + \frac{c(p^m)}{p^{ms}} + \cdots \\
    &= \sum_{n \geq 1} \frac{c_p(n)}{n^s} \quad \text{(say)},
\end{align*}
where
\begin{equation*}
    c_p(n) = 
    \begin{cases}
        1, & \text{if } n=1, \\
        c(n), & \text{if } n = p^m \; (m \geq 2), \\
        0, & \text{otherwise.}
    \end{cases}
\end{equation*}

Note that the above equality holds for $\Re(s) > 1+2\epsilon$ for all $\epsilon > 0$, and that the series is absolutely convergent in this region.
        Also note that $c_p(n) \ll n^\epsilon$ for all $\epsilon > 0$.
        We define $c(n)$ for any $n \in \mathbb{N}$ by $$\prod_{p}\frac{1+A}{1+B} = \prod_p\left( 1+ \sum_{m \geq 1}\frac{c(p^m)}{p^{ms}}\right) = \sum_{n=1}^{\infty}\frac{c(n)}{n^s}.$$
        By construction, $c(n)$ is multiplicative.

     Define \begin{align*}
             H_j^{(1)}(s) : = \frac{F_j^{(1)}(s)}{G_j^{(1)}(s)} &=\prod_{p} \frac{1+ \sum_{m \geq 1}\frac{\lambda_{\mathrm{sym}^jf}(p^m)l_1(p^m)}{p^{ms}}}{1+\sum_{m\geq1}\frac{b(p^m)}{p^{ms}}} \\
             &= \prod_p \frac{1+A}{1+B} = \sum_{n=1}^{\infty}\frac{c(n)}{n^s}.
         \end{align*}

         We now find the region of convergence for $H_j^{(1)}(s)$. Note that \begin{align*}
             \sum_{m \geq 3}\bigg|\frac{c(p^m)}{p^{ms}}\bigg| &\leq \sum_{m\geq 3}\frac{p^{m\epsilon}}{p^{m\sigma}} = \sum_{m \geq 3} \frac{1}{p^{m(\sigma - \epsilon)}} \\
             &= \frac{1}{p^{2(\sigma-\epsilon)}(p^{\sigma -\epsilon}-1)} < \frac{1}{p^{2(\sigma-\epsilon)}} \text{ for any } \epsilon >0.
         \end{align*} The above inequality of the series is true for $\Re(s)> 1+\epsilon$, and\begin{align*}
             \frac{c(p^2)}{p^{2\sigma}} = \frac{\lambda_{\mathrm{sym}^jf}(p^2)l(p^2) - b(p^2)}{p^{2\sigma}} 
             = O\left(\frac{p^{2\epsilon}}{p^{2\sigma}}\right) = O\left(\frac{1}{p^{2\sigma-2\epsilon}}\right).
         \end{align*}

         Now $\prod_p\left(1 + \bigg|\frac{c(p^2)}{p^{2s}}\bigg|+ \sum_{m\geq 3} \bigg|\frac{c(p^m)}{p^{ms}}\bigg|\right) = \prod_p\left(1+u_p\right)$ is convergent if and only if $\sum_pu_p$ is convergent, where $u_p = |\frac{c(p^2)}{p^{2s}}|+\sum_{m \geq 3}|\frac
         {c(p^m)}{p^{ms}}|$. Note that $$ \sum_pu_p \ll  \sum_p \frac{1}{p^{2\sigma -2\epsilon}} $$ is absolutely convergent for $2\sigma -2\epsilon > 1$, that is, in the region $\sigma > \frac{1}{2}+\epsilon$ for any $\epsilon >0$. So in this region $H_j^{(1)}(s) \ll_\epsilon 1$ and $H_j^{(1)}(s)$ is absolutely convergent in $\Re(s) > \frac{1}{2}$. 
\end{proof}
The proofs of the following lemmas proceed along the same lines as those in~\cite{Kaur2026, SharmaSankaranarayanan2023} and are therefore omitted.
\begin{lemma} \label{lem:F2}
Let \(f\) be a normalised primitive holomorphic cusp form of weight \(k\) for \(\operatorname{SL}(2,\mathbb{Z})\), 
and let \(\lambda_{\operatorname{sym}^{j}f}(n)\) denote the \(n\)th normalised Fourier coefficient 
of the \(j\)th symmetric power \(L\)-function attached to \(f\).  Define
\[
F_{j}^{(2)}(s)=\sum_{n=1}^{\infty}\frac{\lambda_{\operatorname{sym}^{j}f}(n)\,l_2(n)}{n^{s}}, \qquad \Re(s)>2,
\]
where \(l_2(n)\) is given by \eqref{eq:l_2(n)}. Then \(F_{j}^{(2)}(s)\) admits a factorisation
\[
F_{j}^{(2)}(s)=G_{j}^{(2)}(s)\,H_{j}^{(2)}(s),
\]
in which
\[
G_{j}^{(2)}(s):=L\!\bigl(s,\operatorname{sym}^{j}f\bigr)\; 
        L\!\bigl(s-1,\operatorname{sym}^{j}f\bigr),
\]
and \(H_{j}^{(2)}(s)\) is a Dirichlet series converging absolutely and uniformly 
in the half-plane \(\Re(s)>\frac{3}{2}\).
\end{lemma}

\begin{lemma} \label{lem:F3'}
Let \(f\) be a normalised primitive holomorphic cusp form of weight \(k\) for \(\operatorname{SL}(2,\mathbb{Z})\), 
and let \(\lambda_{\operatorname{sym}^{j}f}(n)\) denote the \(n\)th normalised Fourier coefficient 
of the \(j\)th symmetric power \(L\)-function attached to \(f\). Define
\[
F_{j_1}^{(3)}(s)=\sum_{n=1}^{\infty}\frac{\lambda_{\operatorname{sym}^{j}f}(n)\,l_3(n)}{n^{s}}, \qquad \Re(s)>3,
\]
where \(l_3(n)\) is given by \eqref{eq:l_3(n)&v_3(n)}.  Then \(F_{j_1}^{(3)}(s)\) admits a factorisation
\[
F_{j_1}^{(3)}(s)=G_{j_1}^{(3)}(s)\,H_{j_1}^{(3)}(s),
\]
in which
\[
G_{j_1}^{(3)}(s):=L\!\bigl(s,\operatorname{sym}^{j}f \otimes \chi_4\bigr)\; 
        L\!\bigl(s-2,\operatorname{sym}^{j}f\bigr),
\]
and \(H_{j_1}^{(3)}(s)\) is a Dirichlet series converging absolutely and uniformly 
in the half-plane \(\Re(s)>\frac{5}{2}\).
\end{lemma}

\begin{lemma} \label{lem:F3''}
Let \(f\) be a normalised primitive holomorphic cusp form of weight \(k\) for \(\operatorname{SL}(2,\mathbb{Z})\), 
and let \(\lambda_{\operatorname{sym}^{j}f}(n)\) denote the \(n\)th normalised Fourier coefficient 
of the \(j\)th symmetric power \(L\)-function attached to \(f\). Define
\[
F_{j_2}^{(3)}(s)=\sum_{n=1}^{\infty}\frac{\lambda_{\operatorname{sym}^{j}f}(n)\,v_3(n)}{n^{s}}, \qquad \Re(s)>3,
\]
where \(v_3(n)\) is given by \eqref{eq:l_3(n)&v_3(n)}. Then \(F_{j_2}^{(3)}(s)\) admits a factorisation
\[
F_{j_2}^{(3)}(s)=G_{j_2}^{(3)}(s)\,H_{j_2}^{(3)}(s),
\]
in which
\[
G_{j_2}^{(3)}(s):=L\!\bigl(s-2,\operatorname{sym}^{j}f \otimes \chi_4\bigr)\; 
        L\!\bigl(s,\operatorname{sym}^{j}f\bigr),
\]
and \(H_{j_2}^{(3)}(s)\) is a Dirichlet series converging absolutely and uniformly 
in the half-plane \(\Re(s)>\frac{5}{2}\).
\end{lemma}

\begin{lemma} \label{lem:F4}
Let \(f\) be a normalised primitive holomorphic cusp form of weight \(k\) for \(\operatorname{SL}(2,\mathbb{Z})\), 
and let \(\lambda_{\operatorname{sym}^{j}f}(n)\) denote the \(n\)th normalised Fourier coefficient 
of the \(j\)th symmetric power \(L\)-function attached to \(f\).  Define
\[
F_{j}^{(4)}(s)=\sum_{n=1}^{\infty}\frac{\lambda_{\operatorname{sym}^{j}f}(n)\,l_4(n)}{n^{s}}, \qquad \Re(s)>4,
\]
where \(l_4(n)\) is given by \eqref{eq:l_4(n)}. Then \(F_{j}^{(4)}(s)\) admits a factorisation
\[
F_{j}^{(4)}(s)=G_{j}^{(4)}(s)\,H_{j}^{(4)}(s),
\]
in which
\[
G_{j}^{(4)}(s):=L\!\bigl(s,\operatorname{sym}^{j}f\bigr)\; 
        L\!\bigl(s-3,\operatorname{sym}^{j}f\bigr),
\]
and \(H_{j}^{(4)}(s)\) is a Dirichlet series converging absolutely and uniformly 
in the half-plane \(\Re(s)>\frac{7}{2}\).
\end{lemma}

\begin{lemma}\label{lem:F5'}
Let \(f\) be a normalised primitive holomorphic cusp form of weight \(k\) for \(\operatorname{SL}(2,\mathbb{Z})\), 
and let \(\lambda_{\operatorname{sym}^{j}f}(n)\) denote the \(n\)th normalised Fourier coefficient 
of the \(j\)th symmetric power \(L\)-function attached to \(f\). Define
\[
F_{j_1}^{(5)}(s)=\sum_{n=1}^{\infty}\frac{\lambda_{\operatorname{sym}^{j}f}(n)\,l_5(n)}{n^{s}} \qquad \Re(s)>5,
\]
where \(l_5(n)\) is given by \eqref{eq:l_5(n)&v_5(n)}. Then \(F_{j_1}^{(5)}(s)\) admits a factorisation
\[
F_{j_1}^{(5)}(s)=G_{j_1}^{(5)}(s)\,H_{j_1}^{(5)}(s),
\]
in which
\[
G_{j_1}^{(5)}(s):=L\!\bigl(s-4,\operatorname{sym}^{j}f\bigr)\; 
        L\!\bigl(s,\operatorname{sym}^{j}f\otimes\chi_4\bigr)
\]
and \(H_{j_1}^{(5)}(s)\) is a Dirichlet series converging absolutely and uniformly 
in the half-plane \(\Re(s)>\frac{9}{2}\).
\end{lemma}

\begin{lemma}\label{lem:F5''}
        Let $f$ be a normalised primitive holomorphic cusp form of weight $k$ for $SL(2,\mathbb{Z})$ and let $\lambda_{\mathrm{sym}^{j}f}(n)$ be the $n$th normalised Fourier coefficient of the $j$th symmetric power $L$-function associated to $f$. Define
        \[
        F_{j_2}^{(5)}(s)=\sum_{n=1}^{\infty}\frac{\lambda_{\mathrm{sym}^{j}f}(n)v_5(n)}{n^{s}}, \quad \Re(s) > 5,
        \]
        where $v_5(n)$ is given by \eqref{eq:l_5(n)&v_5(n)}. Then
        \[
        F_{j_2}^{(5)}(s)=G_{j_2}^{(5)}(s)H_{j_2}^{(5)}(s),
        \]
        where
        \[
        G_{j_2}^{(5)}(s):=L(s, \mathrm{sym}_{j}f)L(s-4,\mathrm{sym}_jf \otimes\chi_4)
        \]
        and $H_{j_2}^{(5)}(s)$ is a Dirichlet series that converges uniformly and absolutely in the half plane $\Re(s)>\frac{9}{2}$.
    \end{lemma}

\begin{lemma}\label{lem:F6}
Let \(f\) be a normalised primitive holomorphic cusp form of weight \(k\) for \(\operatorname{SL}(2,\mathbb{Z})\) and let \(\lambda_{\operatorname{sym}^{j}f}(n)\) denote the \(n\)th normalised Fourier coefficient 
of the \(j\)th symmetric power \(L\)-function attached to \(f\). Define
\[
F_{j}^{(6)}(s)=\sum_{n=1}^{\infty}\frac{\lambda_{\operatorname{sym}^{j}f}(n)\,l_6(n)}{n^{s}}, \qquad \Re(s)>6,
\]
where \(l_6(n)\) is a given by \eqref{eq:l_6(n)}. Then \(F_{j}^{(6)}(s)\) admits a factorisation
\[
F_{j}^{(6)}(s)=G_{j}^{(6)}(s)\,H_{j}^{(6)}(s),
\]
in which 
\[
G_{j}^{(6)}(s):=L\!\bigl(s-5,\operatorname{sym}^{j}f\bigr)\; 
        L\!\bigl(s,\operatorname{sym}^{j}f\bigr)
\]
and \(H_{j}^{(6)}(s)\) is a Dirichlet series converging absolutely and uniformly 
in the half-plane \(\Re(s)>\frac{11}{2}\).
\end{lemma}

\begin{lemma} \label{lem:F0*}
Let \(f\) be a normalised primitive holomorphic cusp form of weight \(k\) for \(\operatorname{SL}(2,\mathbb{Z})\)
and let \(\lambda_{\operatorname{sym}^{j}f}(n)\) denote the \(n\)th normalised Fourier coefficient 
of the \(j\)th symmetric power \(L\)-function attached to \(f\). Define
\[
F_{j}^{*(0)}(s)=\sum_{n=1}^{\infty}\frac{\lambda_{\operatorname{sym}^{j}f}^2(n)\,}{n^{s}}, \qquad \Re(s)>1.
\]  
Then \(F_{j}^{*(1)}(s)\) admits a factorisation
\[
F_{j}^{*(0)}(s)=G_{j}^{*(0)}(s)\,H_{j}^{*(0)}(s),
\]
in which
\[
G_{j}^{*(0)}(s):=\zeta(s)\prod_{n=1}^{j}L\!\bigl(s,\operatorname{sym}^{2n}f \bigr),
\] 
and \(H_{j}^{*(0)}(s)\) is a Dirichlet series converging absolutely and uniformly 
in the half-plane \(\Re(s)>\frac{1}{2}\).
\end{lemma}

\begin{lemma} \label{lem:F1*}
Let \(f\) be a normalised primitive holomorphic cusp form of weight \(k\) for \(\operatorname{SL}(2,\mathbb{Z})\)
and let \(\lambda_{\operatorname{sym}^{j}f}(n)\) denote the \(n\)th normalised Fourier coefficient 
of the \(j\)th symmetric power \(L\)-function attached to \(f\). Define
\[
F_{j}^{*(1)}(s)=\sum_{n=1}^{\infty}\frac{\lambda_{\operatorname{sym}^{j}f}^2(n)\,l_1(n)}{n^{s}}, \qquad \Re(s)>1,
\]
where \(l_1(n)\) is given by \eqref{eq:l_1(n)}. Then \(F_{j}^{*(1)}(s)\) admits a factorisation
\[
F_{j}^{*(1)}(s)=G_{j}^{*(1)}(s)\,H_{j}^{*(1)}(s),
\]
in which
\[
G_{j}^{*(1)}(s):=\zeta(s)L(s,\chi_4)\prod_{n=1}^{j}L\!\bigl(s,\operatorname{sym}^{2n}f \bigr)\; 
        L\!\bigl(s,\operatorname{sym}^{2n}f\otimes\chi_4\bigr),
\]
and \(H_{j}^{*(1)}(s)\) is a Dirichlet series converging absolutely and uniformly 
in the half-plane \(\Re(s)>\frac{1}{2}\).
\end{lemma}

\begin{lemma}[{\cite{Hua2022}}]\label{lem:F2*}
Let \(f\) be a normalised primitive holomorphic cusp form of weight \(k\) for \(\operatorname{SL}(2,\mathbb{Z})\)
and let \(\lambda_{\operatorname{sym}^{j}f}(n)\) denote the \(n\)th normalised Fourier coefficient 
of the \(j\)th symmetric power \(L\)-function attached to \(f\). Define
\[
F_{j}^{*(2)}(s)=\sum_{n=1}^{\infty}\frac{\lambda_{\operatorname{sym}^{j}f}^2(n)\,l_2(n)}{n^{s}}, \qquad \Re(s)>2,
\]
where \(l_2(n)\) is given by \eqref{eq:l_2(n)}. Then \(F_{j}^{*(2)}(s)\) admits a factorisation
\[
F_{j}^{(2)}(s)=G_{j}^{*(2)}(s)\,H_{j}^{*(2)}(s),
\]
in which
\[
G_{j}^{*(2)}(s):=\zeta(s)\zeta(s-1)\prod_{n=1}^{j}L\!\bigl(s,\operatorname{sym}^{2n}f \bigr)\; 
        L\!\bigl(s-1,\operatorname{sym}^{2n}f\bigr)
\]
and \(H_{j}^{*(2)}(s)\) is a Dirichlet series converging absolutely and uniformly 
in the half-plane \(\Re(s)>\frac{3}{2}\).
\end{lemma}

\begin{lemma}[{\cite{SharmaSankaranarayanan2023}}]\label{lem:F3*'}
Let \(f\) be a normalised primitive holomorphic cusp form of weight \(k\) for \(\operatorname{SL}(2,\mathbb{Z})\), 
and let \(\lambda_{\operatorname{sym}^{j}f}(n)\) denote the \(n\)th normalised Fourier coefficient 
of the \(j\)th symmetric power \(L\)-function attached to \(f\). Define
\[
F_{j_1}^{*(3)}(s)=\sum_{n=1}^{\infty}\frac{\lambda_{\operatorname{sym}^{j}f}^2(n)\,l_3(n)}{n^{s}}, \qquad \Re(s)>3,
\]
where \(l_3(n)\) is given by \eqref{eq:l_3(n)&v_3(n)}. Then \(F_{j_1}^{*(3)}(s)\) admits a factorisation
\[
F_{j_1}^{(3)}(s)=G_{j_1}^{*(3)}(s)\,H_{j_1}^{*(3)}(s),
\]
in which
\[
G_{j_1}^{*(3)}(s):=L(s,\chi_4)\zeta(s-2)\prod_{n=1}^{j}L\!\bigl(s,\operatorname{sym}^{2n}f \otimes \chi_4\bigr)\; 
        L\!\bigl(s-2,\operatorname{sym}^{2n}f\bigr)
\]
and \(H_{j_1}^{*(3)}(s)\) is a Dirichlet series converging absolutely and uniformly 
in the half-plane \(\Re(s)>\frac{5}{2}\).
\end{lemma}

\begin{lemma}[{\cite{SharmaSankaranarayanan2023}}]\label{lem:F3*''}
Let \(f\) be a normalised primitive holomorphic cusp form of weight \(k\) for \(\operatorname{SL}(2,\mathbb{Z})\)
and let \(\lambda_{\operatorname{sym}^{j}f}(n)\) denote the \(n\)th normalised Fourier coefficient 
of the \(j\)th symmetric power \(L\)-function attached to \(f\). Define
\[
F_{j_2}^{*(3)}(s)=\sum_{n=1}^{\infty}\frac{\lambda_{\operatorname{sym}^{j}f}^2(n)\,v_3(n)}{n^{s}}, \qquad \Re(s)>3,
\]
where \(v_3(n)\) is given by \eqref{eq:l_3(n)&v_3(n)}. Then \(F_{j_2}^{*(3)}(s)\) admits a factorisation
\[
F_{j_2}^{(3)}(s)=G_{j_2}^{*(3)}(s)\,H_{j_2}^{*(3)}(s),
\]
in which
\[
G_{j_2}^{*(3)}(s):=L(s-2,\chi_4)\zeta(s)\prod_{n=1}^{j}L\!\bigl(s-2,\operatorname{sym}^{2n}f \otimes \chi_4\bigr)\; 
        L\!\bigl(s,\operatorname{sym}^{2n}f\bigr)
\]
and \(H_{j_2}^{*(3)}(s)\) is a Dirichlet series converging absolutely and uniformly 
in the half-plane \(\Re(s)>\frac{5}{2}\).
\end{lemma}

\begin{lemma}[{\cite{PanWang2025}}]\label{lem:F4*}
Let \(f\) be a normalised primitive holomorphic cusp form of weight \(k\) for \(\operatorname{SL}(2,\mathbb{Z})\) 
and let \(\lambda_{\operatorname{sym}^{j}f}(n)\) denote the \(n\)th normalised Fourier coefficient 
of the \(j\)th symmetric power \(L\)-function attached to \(f\). Define
\[
F_{j}^{*(4)}(s)=\sum_{n=1}^{\infty}\frac{\lambda_{\operatorname{sym}^{j}f}^2(n)\,l_4(n)}{n^{s}}, \qquad \Re(s)>4,
\]
where \(l_4(n)\) is given by \eqref{eq:l_4(n)}. Then \(F_{j}^{*(4)}(s)\) admits a factorisation
\[
F_{j}^{(4)}(s)=G_{j}^{*(4)}(s)\,H_{j}^{*(4)}(s),
\]
in which
\[
G_{j}^{*(4)}(s):=\zeta(s)\zeta(s-3)\prod_{n=1}^{j}L\!\bigl(s,\operatorname{sym}^{2n}f\bigr)\; 
        L\!\bigl(s-3,\operatorname{sym}^{2n}f\bigr)
\]
and \(H_{j}^{*(2)}(s)\) is a Dirichlet series converging absolutely and uniformly 
in the half-plane \(\Re(s)>\frac{7}{2}\).
\end{lemma}

The next three lemmas can be proved similarly to those in~\cite{Kaur2026, PanWang2025, SharmaSankaranarayanan2023}.

\begin{lemma} \label{lem:F5*'}
Let $f$ be a normalised primitive holomorphic cusp form of weight $k$ for $SL(2,\mathbb{Z})$ and let $\lambda_{\mathrm{sym}^{j}f}(n)$ be the $n$th normalised Fourier coefficient of the $j$th symmetric power $L$-function associated to $f$. Define
\[
F_{j_1}^{*(5)}(s)=\sum_{n=1}^{\infty}\frac{\lambda_{\mathrm{sym}^{j}f}^{2}(n)l_5(n)}{n^{s}}, \quad \Re(s) > 5,
\]
where \(l_5(n)\) is given by \eqref{eq:l_5(n)&v_5(n)}. Then $F_{j_1}^{*(5)}(s)$ admits a factorisation
\[
F_{j_1}^{*(5)}(s)=G_{j_1}^{*(5)}(s)H_{j_1}^{*(5)}(s),
\]
in which
\[
G_{j_1}^{*(5)}(s):=\zeta(s-4)L(s,\chi_4)\prod_{n=1}^{j}L(s-4,\mathrm{sym}^{2n}f)L(s,\mathrm{sym}^{2n}f\otimes\chi_4)
\]
and $H_{j_1}^{*(5)}(s)$ is a Dirichlet series that converges uniformly and absolutely in the half plane $\Re(s)>\frac{9}{2}$. 
\end{lemma}

\begin{lemma} \label{lem:F5*''}
Let $f$ be a normalised primitive holomorphic cusp form of weight $k$ for $SL(2,\mathbb{Z})$, and let $\lambda_{\mathrm{sym}^{j}f}(n)$ be the $n$th normalised Fourier coefficient of the $j$th symmetric power $L$-function associated to $f$. Define
\[
F_{j_2}^{*(5)}(s)=\sum_{n=1}^{\infty}\frac{\lambda_{\mathrm{sym}^{j}f}^{2}(n)v_5(n)}{n^{s}}, \quad \Re(s) > 5,
\]
where \(v_5(n)\) is given by \eqref{eq:l_5(n)&v_5(n)}. Then $F_{j_2}^{*(5)}(s)$ admits a factorisation
\[
F_{j_2}^{*(5)}(s)=G_{j_2}^{*(5)}(s)H_{j_2}^{*(5)}(s),
\]
in which
\[
G_{j_2}^{*(5)}(s):=\zeta(s)L(s-4,\chi_4)\prod_{n=1}^{j}L(s,\mathrm{sym}^{2n}f)L(s-4,\mathrm{sym}^{2n}f\otimes\chi_4),
\]
and $H_{j_2}^{*(5)}(s)$ is a Dirichlet series that converges uniformly and absolutely in the half plane $\Re(s)>\frac{9}{2}$.
\end{lemma}

\begin{lemma} \label{lem:F6*}
Let $f$ be a normalised primitive holomorphic cusp form of weight $k$ for $SL(2,\mathbb{Z})$, and let $\lambda_{\mathrm{sym}^{j}f}(n)$ be the $n$th normalised Fourier coefficient of the $j$th symmetric power $L$-function associated to $f$. Define
\[
F_{j}^{*(6)}(s)=\sum_{n=1}^{\infty}\frac{\lambda_{\mathrm{sym}^{j}f}^{2}(n)l_6(n)}{n^{s}}, \quad \Re(s) > 6,
\]
where \(l_6(n)\) is given by \eqref{eq:l_6(n)}. Then $F_{j}^{*(6)}(s)$ admits a factorisation
\[
F_{j}^{*(6)}(s)=G_{j}^{*(6)}(s)H_{j}^{*(6)}(s),
\]
in which
\[
G_{j}^{*(6)}(s):=\zeta(s-5)\zeta(s)\prod_{n=1}^{j}L(s-5,\mathrm{sym}^{2n}f)L(s,\mathrm{sym}^{2n}f),
\]
and $H_{j}^{*(6)}(s)$ is a Dirichlet series that converges uniformly and absolutely in the half plane $\Re(s)>\frac{11}{2}$. 
\end{lemma}

\begin{lemma} \label{lem:zetabound}
For $T \geq 2$, we have
\begin{align} \label{eq:zeta12mom}
\int_{1}^{T}\left|\zeta\left(\frac{1}{2}+\epsilon+it\right)\right|^{4}dt\ll_\epsilon T^{1+\epsilon},
\end{align}
and for any $\epsilon > 0$,
\begin{align} \label{eq:zetabound}
\zeta(\sigma+it)\ll_{\epsilon}(1+|t|)^{\max\{\frac{13}{42}(1-\sigma),0\}+\epsilon}
\end{align}
uniformly for $\frac{1}{2}\leq\sigma\leq 1+\epsilon$, $|t|\geq 1$.
\end{lemma}
\begin{proof}

By~\cite[p.~148]{Titch}, we have \begin{align} \label{eq:ZetaIntBound}
\int_{1}^{T}\left|\zeta\left(\frac{1}{2}+it\right)\right|^{4}dt\ll T^{1+\epsilon}
\end{align} for any $\epsilon > 0$. Also we have \begin{align} \label{eq:zetaLocalBound}
    |\zeta(1+it)| \ll_{\delta} |t|^{\delta},
\end{align}
for any $\delta > 0 $, and $|t| \geq 1$. Now, for a fixed $t \geq 1$ and by Hadamard's three-lines principle, we have \begin{align}
    |\zeta(\frac{1}{2}+\epsilon+it)| \ll |\zeta(\frac{1}{2}+it)|^{1-2\epsilon} |\zeta(1+it)|^{2\epsilon}.
\end{align}

If we take $\epsilon < \frac{1}{2}$, then $1-2\epsilon < 1$. Now using \eqref{eq:ZetaIntBound} and \eqref{eq:zetaLocalBound}, we have \begin{align}
    \int_{1}^{T}  |\zeta(\frac{1}{2}+\epsilon+it)|^4 &\ll T^{\epsilon_1}\int_{1}^{T}|\zeta(\frac{1}{2}+it)|^{4(1-2\epsilon)} \\
    & \ll T^{\epsilon_1 } \left(\int_{1}^{T}|\zeta(\frac{1}{2}+it)|^{4} \right)^{1-2\epsilon} \left( \int_{1}^{T} 1 dt \right)^{2\epsilon} \\
    & \ll T^{\epsilon_1} T^{(1+\epsilon)(1-2\epsilon)}T^{2\epsilon} \\
    & \ll T^{1+\epsilon}.
\end{align}

On the other hand, from~\cite{BourJe}, we have \begin{align} \label{eq:ZetaHalfBound}
    \zeta\left( \frac{1}{2}+it \right) \ll |t|^{\frac{13}{84}+\epsilon}.
\end{align}
For $\frac{1}{2} \leq \sigma \leq 1$, using Hadamard's three-lines principle, \eqref{eq:zetaLocalBound}, and \eqref{eq:ZetaHalfBound}, we have \begin{align}
    \zeta\left( \sigma +it \right) &\ll \zeta\left( \frac{1}{2}+it \right)^{\frac{1-\sigma}{1-\frac{1}{2}}}\zeta\left( 1+it \right)^{\frac{\sigma -\frac{1}{2}}{1-\frac{1}{2}}} \ll |t|^{\frac{13}{42}(1-\sigma)+\epsilon},
\end{align}
and for $1 < \sigma \leq 1+\epsilon $, we have $\zeta\left( \sigma +it \right) \ll_\epsilon 1$, which proves the second result.
\end{proof}

\begin{lemma}\label{lem:LSigmaChi}
    Let $\chi$ be any primitive character modulo $q$. Then for $q \ll T^2$,
    \begin{align}
        L(\sigma+iT,\chi) \ll (q(1+|T|))^{\mathrm{max}\{\frac{1}{3}(1-\sigma),0\}+\epsilon}
    \end{align}
    holds uniformly for $\frac{1}{2} \leq \sigma \leq 2$, and 
    \begin{align}
        \int_{1}^{T}|L(\frac{1}{2}+\epsilon+it,\chi)|^4 \ll_{q} T^{1+\epsilon}.
    \end{align}
\end{lemma}
\begin{proof}
    From~\cite{Heath}, for $|t| \geq 1$, we have \begin{align} \label{eq:LHalfbound}
        L(\frac{1}{2}+it, \chi) \ll t^{\frac{1}{6}+\epsilon}.
    \end{align} 
    Similarly, we have
    \begin{align} \label{eq:L1bound}
        L(1+it, \chi) \ll_\delta |t|^{\delta},
    \end{align}
    for $0 < \delta < 1$. 
    Using Hadamard's three-lines principle, \eqref{eq:LHalfbound}, and \eqref{eq:L1bound}, we have \begin{align}
        L(\sigma+it,\chi) &\ll L(\frac{1}{2}+it,\chi)^{2(1-\sigma)}L(1+it, \chi)^{2\sigma-1} 
        \ll_\delta|t|^{\frac{1}{3}(1-\sigma)+\delta}
    \end{align}
  for $\sigma\in [\frac{1}{2},1]$ and any $\delta>0$.
    For $1 \leq \sigma \leq 2$, we have $L(\sigma+it, \chi) \ll  1$. Thus the first result follows.

    Again, from~\cite{Berke2021}, we have \begin{align} \label{eq:LIntHalfBound}
        \int_{1}^{T}\left |L\left(\frac{1}{2}+it,\chi \right) \right|^4 \ll_{q,\epsilon} T^{1+\epsilon}.
    \end{align}
     We now use Hadamard's three-lines principle, \eqref{eq:L1bound}, and \eqref{eq:LIntHalfBound}. Proceeding as in the Lemma \ref{lem:zetabound}, we have \begin{align}
        L(\frac{1}{2}+\epsilon+it) &\ll L\left( \frac{1}{2}+it, \chi \right)^{1-2\epsilon}L \left( 1+it, \chi \right)^{2\epsilon}        \ll_\delta |t|^{\delta}  L\left( \frac{1}{2}+it, \chi \right)^{1-2\epsilon}.
    \end{align}
    Hence, we have \begin{align}
        \int_{1}^{T} |L(\frac{1}{2}+\epsilon+it, \chi)|^{4}dt &\ll T^{4 \delta}\int_{1}^{T}\left|L\left( \frac{1}{2}+it, \chi \right)\right|^{4(1-2\epsilon)}dt \\
        &\ll T^{4\delta} \left( \int_{1}^{T}\left|L\left( \frac{1}{2}+it, \chi \right)\right|^4 \right)^{1-2\epsilon}\left( \int_{1}^{T}1 dt \right)^{2\epsilon} \\
        &\ll T^{4\delta}T^{(1+\epsilon)(1-2\epsilon)+2\epsilon} \\
        & \ll T^{1+\epsilon},
    \end{align}
    from which the second result follows.
    \end{proof}

\begin{lemma} \label{lem:LSymFChi}
    Let $f \in H_k$, and $\chi$ be a primitive character modulo $q$. Then for $q \ll T^2$, we have 
    \begin{align} \label{eq:LSigmafbound}
        L(\sigma+iT,\mathrm{sym}^2f) \ll (1+|T|)^{\mathrm{max} \{\frac{6}{5}(1-\sigma),0\} + \epsilon},
    \end{align}
    and \begin{align}\label{eq:LSigmaSymChibound}
        L(\sigma+iT,\mathrm{sym}^2f \otimes\chi) \ll (q(1+|T|))^{\mathrm{max} \{\frac{67}{46}(1-\sigma),0\} + \epsilon}
    \end{align}
    uniformly for $\frac{1}{2} \leq \sigma \leq 2$ and $|T| \geq 1$. We also have
    \begin{align}\label{eq:intLSigmafChi}
        \int_{1}^{T}|L(\sigma+iT,\mathrm{sym}^2f \otimes\chi)|^4dt \ll (qT)^{6(1-\sigma)+\epsilon}
    \end{align}
    uniformly for $\frac{1}{2} \leq \sigma \leq 1+\epsilon$ and $T \geq 1$.
 \end{lemma}

 \begin{proof}
     The estimates \eqref{eq:LSigmafbound} and \eqref{eq:LSigmaSymChibound} follow from the Phragm\'{en}-Lindel\"{o}f convexity principle and the works of Lin, Nunes, and Qi~\cite{Lin} and Huang~\cite{Huang}, respectively. The bound \eqref{eq:intLSigmafChi} follows from Perelli~\cite{Per}.
 \end{proof}

\begin{lemma}[{\cite{JiangLu}}]\label{lem:genLbound}
Let $\chi$ be a primitive character modulo $q$ and $\mathfrak{L}_{m,n}^{d}(s,\chi)$ be a general $L$-function of degree $2A$. For any $\epsilon>0$, we have
\begin{align} \label{eq:genL2mom}
\int_{T}^{2T}\left|\mathfrak{L}_{m,n}^{d}(\sigma+it,\chi)\right|^{2}dt\ll (qT)^{2A(1-\sigma)+\epsilon},
\end{align}
uniformly for $\frac{1}{2}\leq\sigma\leq 1+\epsilon$, and $T\geq 1$. Also,
\begin{align} \label{eq:genLbound}
\mathfrak{L}_{m,n}^{d}(\sigma+it,\chi)\ll (q(1+|t|))^{\max\{A(1-\sigma),0\}+\epsilon},
\end{align}
uniformly for $-\epsilon\leq\sigma\leq 1+\epsilon$.
\end{lemma}
\begin{lemma} \label{lem:int}
    Let $f:\mathbb{R} \mapsto \mathbb{R}$ be any function, and $T > 1$. Then \begin{align}
        \int_{1}^{T}\frac{|f(t)|}{t}dt \ll \log T \sup_{1 \leq T_1 \leq T} \frac{1}{T_1}\int_{T_1}^{2T_1}|f(t)|dt.
    \end{align}
\end{lemma}
\begin{proof}
    We note that there exists an $N$ such that $2^N \leq T < 2^{N+1}$. Now  \begin{align}
        \int_{1}^{T}\frac{|f(t)|}{t}dt &\leq \sum_{k=0}^{N}\int_{2^k}^{2^{k+1}}\frac{|f(t)|}{t}dt \leq \sum_{k=0}^{N} \frac{1}{2^k}\int_{2^k}^{2^{k+1}}|f(t)|dt.
    \end{align}
   Let
\begin{equation*}
    M = \sup_{1 \leq T_1 \leq T} \frac{1}{T_1}\int_{T_1}^{2T_1}|f(t)|\,dt,
\end{equation*}
so that
\begin{equation}
    \frac{1}{2^k}\int_{2^k}^{2^{k+1}}|f(t)|\,dt \leq M,
\end{equation}
and
\begin{equation}
    \int_{1}^{T}\frac{|f(t)|}{t}\,dt \leq \sum_{k=0}^{N}M = (N+1)M.
\end{equation}
Since $2^N \leq T < 2^{N+1}$, we have $N \ll \log T$. Thus 
    \begin{align}
        \int_{1}^{T}\frac{|f(t)|}{t}dt &\leq 2NM \ll M\log T 
         \ll \log T \sup_{1\leq T_1 \leq T}\frac{1}{T_1}\int_{T_1}^{2T_1}|f(t)|dt.
    \end{align}
\end{proof}

\begin{lemma}\label{dominating main term}
Let \(f: \mathbb{R} \to \mathbb{R}_{>0}\) be any function that satisfy
\[
f(x) = Dx^{A} + O(x^{B})
\]
as \(x \to \infty\), where \(A, B \in \mathbb{R}\), \( D>0\), and \(A > B\).  
Then there exists \(X_0 > 0\) such that for all \(x > X_0\),
\[
f(x) \ge \frac{D}{2} x^{A}.
\]
\end{lemma}
\begin{proof}
By hypothesis, there exist constants \(C > 0\) and \(X_1 > 0\) such that for all \(x > X_1\),
\[
|f(x) - Dx^{A}| \le C x^{B}.
\]
Hence,
\[
f(x) \ge Dx^{A} - C x^{B}.
\]
Factoring \(x^{A}\), we get
\[
f(x) \ge Dx^{A} \left( 1 - \frac{C}{D} x^{B-A} \right).
\]
Since \(B - A < 0\), we have \(x^{B-A} \to 0\) as \(x \to \infty\).  
We choose \(X_0 \ge X_1\) large enough such that
\[
\frac{C}{D} x^{B-A} \le \frac{1}{2} \quad \text{for all } x > X_0.
\]
Then for all \(x > X_0\),
\[
f(x) \ge Dx^{A} \left( 1 - \frac{1}{2} \right) = \frac{D}{2} x^{A}.
\]
\end{proof}

\section{Proof of Theorem \ref{thm:main(0)}} \label{sec:main0}
By Perron's formula and \eqref{lem:F0}, we have \begin{align}
    \sum_{n\leq x}\lambda_{\mathrm{sym}^{j}f}(n) & = \int_{1+\epsilon-iT}^{1+\epsilon+iT}F_{j}^{(0)}(s)\frac{x^s}{s}ds + O\left( \frac{x^{1+\epsilon}}{T}\right).
\end{align}
We move the line of integration to $\Re(s) = \frac{1}{2}+\epsilon$ and by the Cauchy residue theorem, we get that there exists no pole in the area of integration due to the Lemma \ref{lem:F0}.
\begin{align} \label{eq:sum_l0}
    \sum_{n\leq x}\lambda_{\mathrm{sym}^{j}f}(n) & = \frac{1}{2\pi i}\left\{\int_{\frac{1}{2}+\epsilon-iT}^{\frac{1}{2}+\epsilon+iT}+\int_{1+\epsilon-iT}^{\frac{1}{2}+\epsilon-iT}+\int_{\frac{1}{2}+\epsilon+iT}^{1+\epsilon+iT}\right\}F_{j}^{(0)}(s)\frac{x^{s}}{s}ds + O\left(\frac{x^{1+\epsilon}}{T}\right) \\
    &= \frac{1}{2\pi i}(J_{1}+J_{2}+J_{3})+O\left(\frac{x^{1+\epsilon}}{T}\right), \quad \text{(say)} .
\end{align}
Contribution of horizontal line integrals ($J_{2}$ and $J_{3}$) in absolute value (using Lemmas \ref{lem:F0} and \ref{lem:genLbound}) is given by
\begin{align*}
|J_2+J_3| &= \left| \left( \int_{1+\epsilon-iT}^{\frac{1}{2}+\epsilon-iT}+\int_{\frac{1}{2}+\epsilon+iT}^{1+\epsilon+iT} \right)F_j^{(0)}(s)\frac{x^s}{s}\right|\ll  \left( \int_{\frac{1}{2}+\epsilon}^{1+\epsilon} + \int_{\frac{1}{2}+\epsilon}^{1+\epsilon}\right) \frac{|L(\sigma+iT, \mathrm{sym}^j f)|}{T}x^\sigma d\sigma .
\end{align*}
Thus, 
\begin{align*}
    J_2+J_3 &\ll \int_{\frac{1}{2}+\epsilon}^{1+\epsilon} \frac{|L(\sigma+iT , \mathrm{sym}^j f)|}{T}x^{\sigma}d\sigma \\
&\ll \int_{\frac{1}{2}+\epsilon}^{1+\epsilon} \frac{|T|^{\frac{j+1}{2}(1-\sigma)+\epsilon}}{T}x^{\sigma}d\sigma \quad \text{ (by \ref{lem:genLbound})} \\
&\ll \frac{1}{T} \max_{\frac{1}{2}+\epsilon < \sigma < 1+\epsilon} \left( x^\sigma T^{\frac{j+1}{2}(1-\sigma)+\epsilon}\right) .
\end{align*}
Clearly, $x^\sigma T^{\frac{j+1}{2}(1-\sigma)+\epsilon}$ is a monotonic function, so the maximum occurs at the end points of the interval. We take values at both extreme points of the interval $[\frac{1}{2}+\epsilon, 1+\epsilon]$. So 
\begin{align*}
    J_2+J_3 &\ll \frac{1}{T}\left( x^{\frac{1}{2}+\epsilon}T^{\frac{j+1}{2} (1-\frac{1}{2}-\epsilon) +\epsilon} \right) +  \frac{1}{T}\left( x^{1+\epsilon}T^{\frac{j+1}{2} (1-1-\epsilon) +\epsilon}
    \right) \\
    &\ll \frac{x^{1+\epsilon}}{T}+x^{\frac{1}{2}+\epsilon}T^{\frac{j+1}{4}-1+\epsilon}.
\end{align*}

\begin{align*}
    J_1 &= \int_{\frac{1}{2}+\epsilon - iT}^{\frac{1}{2}+\epsilon+iT}F_j^{(0)}(s) \frac{x^{\frac{1}{2}+\epsilon+it}}{\frac{1}{2}+\epsilon+it}ds \\
    &= x^{\frac{1}{2}+\epsilon}\left( \int_{0\leq |t|\leq 1} + \int_{1 \leq |t| \leq T} \right)F_j^{(0)}\left(\frac{1}{2}+\epsilon+it\right) \frac{x^{it}}{\frac{1}{2}+\epsilon+it}idt\\
    &= I_1+I_2.
\end{align*}

Now \begin{align*}
    I_2 &\ll x^{\frac{1}{2}+\epsilon} \int_{1}^{T}|L(\frac{1}{2}+\epsilon+it, \mathrm{sym}^j f)| \frac{1}{t} dt \quad ( \text{ by Lemma } \ref{lem:F0})\\
    &\ll x^{\frac{1}{2}+\epsilon} \log T\text{ }\sup_{1 \leq T_1 \leq T}\frac{1}{T_1}\int_{T_1}^{2T_1}|L(\frac{1}{2}+\epsilon+it, \mathrm{sym}^j f)| dt ( \text{ by Lemma } \ref{lem:int}  \\
    & \ll x^{\frac{1}{2}+\epsilon}\log T\text{ }\sup_{1 \leq T_1 \leq T}\frac{1}{T_1} \left( \int_{T_1}^{2T_1}|L(\frac{1}{2}+\epsilon+it, \mathrm{sym}^j f)|^2 dt \right)^{\frac{1}{2}}\left( \int_{T_1}^{2T_1} 1 dt \right)^{\frac{1}{2}} \\
    & \ll x^{\frac{1}{2}+\epsilon} \sup_{1 \leq T_1 \leq T}\frac{1}{T_1}(T_1^{\max\{(j+1)(1-\frac{1}{2}-\epsilon), 0\} + \epsilon})^{\frac{1}{2}}T_1^\frac{1}{2} \quad \text{(using \ref{lem:genLbound})} \\
    & \ll x^{\frac{1}{2}+\epsilon}T^{\frac{j+1}{4}-\frac{1}{2}+\epsilon}.
\end{align*}

The first integral gives \begin{align*}
    I_1 &= x^{\frac{1}{2}+\epsilon}\int_{0 \leq |t| \leq 1}F_j^{(0)}\left(\frac{1}{2}+\epsilon+it\right) \frac{x^{it}}{\frac{1}{2}+\epsilon+it}dt.
    \end{align*}
    The above integration is finite. If not, then \eqref{eq:sum_l0} would be infinite. As the other integral is finite, this is a contradiction. So,
    \begin{align*}
        I_1 \ll x^{\frac{1}{2}+\epsilon}.
    \end{align*}
Combining $I_1$ and $I_2$, we have \begin{align}
    J_1 \ll x^{\frac{1}{2}+\epsilon} + x^{\frac{1}{2}+\epsilon} T^{\frac{j+1}{4}-\frac{1}{2}+\epsilon}.
\end{align}

Thus we have \begin{align}
    J_1+J_2+J_3 \ll \frac{x^{1+\epsilon}}{T}+x^{\frac{1}{2}+\epsilon} T^{\frac{j+1}{4}-\frac{1}{2}+\epsilon}.
\end{align}

Now put $T = x^{\frac{2}{j+3}}$, then we have \begin{align}
    \sum_{\substack{n\leq x}}\lambda_{\mathrm{sym}^jf}(n) &= O\left( x^{1-\frac{2}{j+3}+\epsilon} \right).
\end{align}

\section{Proof of Theorem \ref{thm:main}} \label{sec:main}

We begin by applying Perron's formula to $F_{j}^{(1)}(s)$ with $\eta=1+\epsilon$, and $10\leq T\leq x$. Thus we have,
\begin{align*}
\sum_{n\leq x}\lambda_{\mathrm{sym}^{j}f}(n)l_1(n)
&= \frac{1}{2\pi i}\int_{\eta-iT}^{\eta+iT}F_{j}^{(1)}(s)\frac{x^{s}}{s}ds + O\left(\frac{x^{1+\epsilon}}{T}\right).
\end{align*}

After moving the line of integration to $\Re(s)=\frac{1}{2}+\epsilon$, by Cauchy's residue theorem, there are no poles due to the Lemma \ref{lem:F1}.
So we obtain,

\begin{align} \label{eq:LambdaEq}
    \sum_{n\leq x}\lambda_{\mathrm{sym}^{j}f}(n)l_1(n) &= \frac{1}{2\pi i}\left\{\int_{\frac{1}{2}+\epsilon-iT}^{\frac{1}{2}+\epsilon+iT}+\int_{1+\epsilon-iT}^{\frac{1}{2}+\epsilon-iT}+\int_{\frac{1}{2}+\epsilon+iT}^{1+\epsilon+iT}\right\}F_{j}^{(1)}(s)\frac{x^{s}}{s}ds \\
    &\quad + O\left(\frac{x^{1+\epsilon}}{T}\right) \\
    &= \frac{1}{2\pi i}(J_{1}+J_{2}+J_{3})+O\left(\frac{x^{1+\epsilon}}{T}\right), \quad \text{(say)}.
\end{align}

Contribution of horizontal line integrals ($J_{2}$ and $J_{3}$) in absolute value (using Lemmas \ref{lem:F1} and \ref{lem:genLbound}) is

\begin{align}
    |J_2+J_3| &= \left| \left( \int_{1+\epsilon-iT}^{\frac{1}{2}+\epsilon-iT}+\int_{\frac{1}{2}+\epsilon+iT}^{1+\epsilon+iT} \right)F_j^{(1)}(s)\frac{x^s}{s}\right| \\
    &\ll  \left( \int_{\frac{1}{2}+\epsilon}^{1+\epsilon} + \int_{\frac{1}{2}+\epsilon}^{1+\epsilon}\right) \frac{|L(\sigma+iT, \mathrm{sym}^j f ) L(\sigma+iT, \mathrm{sym}^j f \otimes \chi_4)|}{T}x^\sigma d\sigma ,
\end{align}

\begin{align*}
    J_2+J_3 &\ll \int_{\frac{1}{2}+\epsilon}^{1+\epsilon} \frac{|T|^{\frac{j+1}{2}(1-\sigma)+\epsilon}|T|^{\frac{j+1}{2}(1-\sigma)+\epsilon}}{T}x^{\sigma}d\sigma \quad \text{ (using \ref{lem:genLbound})} \\
    &\ll \frac{1}{T} \max_{\frac{1}{2}+\epsilon < \sigma < 1+\epsilon} \left( x^\sigma T^{(j+1)(1-\sigma)+\epsilon}\right) .
\end{align*}

Clearly, $x^\sigma T^{(j+1)(1-\sigma)+\epsilon}$ is a monotonic function, so the maximum occurs at the end points of the interval. We take values at both extreme points of the interval $[\frac{1}{2}+\epsilon, 1+\epsilon]$. So 
\begin{align*}
    J_2+J_3 &\ll \frac{x}{T}\left( x^{\frac{1}{2}+\epsilon}T^{(j+1) (1-\frac{1}{2}-\epsilon) +\epsilon} \right) +  \frac{x}{T}\left( x^{1+\epsilon}T^{(j+1) (1-1-\epsilon) +\epsilon}
    \right) \\
    &\ll \frac{x^{2+\epsilon}}{T}+x^{\frac{3}{2}+\epsilon}T^{\frac{j+1}{2}-1+\epsilon}.
\end{align*}

Now contribution of vertical line integral $J_1$ in absolute value is 
\begin{align*}
    J_1 &= \int_{\frac{1}{2}+\epsilon - iT}^{\frac{1}{2}+\epsilon+iT}F_j^1(s) \frac{x^{\frac{1}{2}+\epsilon+it}}{\frac{1}{2}+\epsilon+it}ds \\
    &= x^{\frac{1}{2}+\epsilon}\left( \int_{0\leq |t|\leq 1} + \int_{1 \leq |t| \leq T} \right)F_j^{(1)}\left(\frac{1}{2}+\epsilon+it\right) \frac{x^{it}}{\frac{1}{2}+\epsilon+it}idt\\
    &= I_1+I_2.
\end{align*}

Now using the Lemma \ref{lem:genLbound}, we have
\begin{align*}
    I_2 &\ll x^{\frac{1}{2}+\epsilon} \int_{1}^{T}|L(\frac{1}{2}+\epsilon+it, \mathrm{sym}^j f) L(\frac{1}{2}+\epsilon+it, \mathrm{sym}^j f \otimes \chi_4)| \frac{1}{t} dt \\
    &\ll x^{\frac{1}{2}+\epsilon} \log T\text{ }\sup_{1 \leq T_1 \leq T}\frac{1}{T_1}\int_{T_1}^{2T_1}|L(\frac{1}{2}+\epsilon+it, \mathrm{sym}^j f \otimes ) L(\frac{1}{2}+\epsilon+it, \mathrm{sym}^j f \otimes \chi_4)| dt \\
    & \ll x^{\frac{1}{2}+\epsilon}\log T\text{ }\sup_{1 \leq T_1 \leq T}\frac{1}{T_1} \left( \int_{T_1}^{2T_1}|L(\frac{1}{2}+\epsilon+it, \mathrm{sym}^j f)|^2 dt \right)^{\frac{1}{2}} \\
    & \times\left( \int_{T_1}^{2T_1} |L(\frac{1}{2}+\epsilon+it, \mathrm{sym}^j f\otimes \chi_4)|^2 dt \right)^{\frac{1}{2}} \\
    & \ll x^{\frac{1}{2}+\epsilon} \sup_{1 \leq T_1 \leq T}\frac{1}{T_1}\left(T_1^{\max\{(j+1)(1-\frac{1}{2}-\epsilon), 0\} + \epsilon})^{\frac{1}{2}}\right)\left((T_1)^{\max\{(j+1)(1-\frac{1}{2}-\epsilon), 0\} + \epsilon})^{\frac{1}{2}}\right) \quad \text{(using \ref{lem:genLbound})} \\
    & \ll x^{\frac{1}{2}+\epsilon}T^{\frac{j+1}{2}-1+\epsilon}.
\end{align*}

The first integral gives \begin{align*}
    I_1 &= x^{\frac{1}{2}+\epsilon}\int_{0 \leq |t| \leq 1}F_j^{(1)}\left(\frac{1}{2}+\epsilon+it\right) \frac{x^{it}}{\frac{1}{2}+\epsilon+it}dt
    \end{align*}
    The above integration is finite. If not, then the left-hand side of \eqref{eq:LambdaEq} would be infinite. As the other integral is finite, this is a contradiction. So,
    \begin{align*}
        I_1 \ll x^{\frac{1}{2}+\epsilon}.
    \end{align*}
    
    Combining $I_1$ and $I_2$, we have \begin{align}
    J_1 \ll x^{\frac{1}{2}+\epsilon} + x^{\frac{1}{2}+\epsilon} T^{\frac{j+1}{2}-1+\epsilon}.
    \end{align}
    
    Thus we have \begin{align}
    J_1+J_2+J_3 \ll \frac{x^{1+\epsilon}}{T}+x^{\frac{1}{2}+\epsilon} T^{\frac{j+1}{2}-1+\epsilon}.
    \end{align}

So $\sum_{n\leq x}\lambda_{\mathrm{sym}^{j}f}(n)l_1(n) = O \left(\frac{x^{1+\epsilon}}{T}+x^{\frac{1}{2}+\epsilon} T^{\frac{j+1}{2}-1+\epsilon} \right)$.

Taking $T = x^{\frac{1}{j+1}}$, we get
\begin{align}
        \sum_{\substack{n\leq x }}\lambda_{\mathrm{sym}^jf}(n)l_1(n) = O \left(x^{1-\frac{1}{j+1}+\epsilon} \right),
    \end{align}
which together with \eqref{eq:r_2&l_1} completes the proof.

\section{Proof of Theorem \ref{thm:main(468)}} \label{sec:main(468)}

We first consider the sum of $4$ squares. Now By Perron's formula, we have \begin{align}
    \sum_{n\leq x}\lambda_{\mathrm{sym}^{j}f}(n)l_2(n) & = \int_{2+\epsilon-iT}^{2+\epsilon+iT}F_{j}^{(2)}(s)\frac{x^s}{s}ds + O\left( \frac{x^{2+\epsilon}}{T}\right).
\end{align}

We move the line of integration to $\Re(s) = \frac{3}{2}+\epsilon$ and by the Cauchy residue theorem, we get that there exists no pole in the area of integration due to the Lemma \ref{lem:F2}.

\begin{align} \label{eq:sum_l2}
    \sum_{n\leq x}\lambda_{\mathrm{sym}^{j}f}(n)l_2(n) & = \frac{1}{2\pi i}\left\{\int_{\frac{3}{2}+\epsilon-iT}^{\frac{3}{2}+\epsilon+iT}+\int_{2+\epsilon-iT}^{\frac{3}{2}+\epsilon-iT}+\int_{\frac{3}{2}+\epsilon+iT}^{2+\epsilon+iT}\right\}F_{j}^{(2)}(s)\frac{x^{s}}{s}ds \\
    &\quad + O\left(\frac{x^{2+\epsilon}}{T}\right) \\
    &= \frac{1}{2\pi i}(J_{1}+J_{2}+J_{3})+O\left(\frac{x^{2+\epsilon}}{T}\right), \quad \text{(say)}. 
\end{align}
Contribution of horizontal line integrals ($J_{2}$ and $J_{3}$) in absolute value (using Lemmas \ref{lem:F2} and \ref{lem:genLbound}) is
\begin{align*}
|J_2+J_3| &= \left| \left( \int_{2+\epsilon-iT}^{\frac{3}{2}+\epsilon-iT}+\int_{\frac{3}{2}+\epsilon+iT}^{2+\epsilon+iT} \right)F_j^{(2)}(s)\frac{x^s}{s}\right| \\
&\ll  \left( \int_{\frac{3}{2}+\epsilon}^{2+\epsilon} + \int_{\frac{3}{2}+\epsilon}^{2+\epsilon}\right) \frac{|L(\sigma+iT-1, \mathrm{sym}^j f)|}{T}x^\sigma d\sigma .
\end{align*}
The above inequality happens because $F_j^{(2)}(s)=G_j^{(2)}(s)H_j^{(2)}(s)$ and $H_j^{(2)}(s) \ll 1$ for $\Re(s) > \frac{3}{2}.$ So $F_j^{(2)}(s) \ll G_j^{(2)}(s) = L(s-1, \mathrm{sym}^j)L(s, \mathrm{sym}^j f )$. Now $L(s, \mathrm{sym}^j f)$ is absolutely convergent for $\Re(s) > 1$. So $L(s, \mathrm{sym}^j f) \ll 1$ in $\Re(s) > \frac{3}{2}$.
\begin{align*}
    J_2+J_3 &\ll \int_{\frac{3}{2}+\epsilon}^{2+\epsilon}\frac{|L(\sigma-1+iT , \mathrm{sym}^j f)|}{T} 
x^\sigma d\sigma \\
&\ll \int_{\frac{1}{2}+\epsilon}^{1+\epsilon} \frac{|L(\sigma+iT , \mathrm{sym}^j f)|}{T}x^{\sigma+1}d\sigma \\
&\ll \int_{\frac{1}{2}+\epsilon}^{1+\epsilon} \frac{|T|^{\frac{j+1}{2}(1-\sigma)+\epsilon}}{T}x^{\sigma+1}d\sigma \quad \text{ (using \ref{lem:genLbound})} \\
&\ll \frac{x}{T} \max_{\frac{1}{2}+\epsilon < \sigma < 1+\epsilon} \left( x^\sigma T^{\frac{j+1}{2}(1-\sigma)+\epsilon}\right) .
\end{align*}
Clearly, $x^\sigma T^{\frac{j+1}{2}(1-\sigma)+\epsilon}$ is a monotonic function, so the maximum occurs at the end points of the interval. We take values at both extreme points of the interval $[\frac{1}{2}+\epsilon, 1+\epsilon]$. So 
\begin{align*}
    J_2+J_3 &\ll \frac{x}{T}\left( x^{\frac{1}{2}+\epsilon}T^{\frac{j+1}{2} (1-\frac{1}{2}-\epsilon) +\epsilon} \right) +  \frac{x}{T}\left( x^{1+\epsilon}T^{\frac{j+1}{2} (1-1-\epsilon) +\epsilon}
    \right) \\
    &\ll \frac{x^{2+\epsilon}}{T}+x^{\frac{3}{2}+\epsilon}T^{\frac{j+1}{4}-1+\epsilon},
\end{align*}
and
\begin{align*}
    J_1 &= \int_{\frac{3}{2}+\epsilon - iT}^{\frac{3}{2}+\epsilon+iT}F_j^{(2)}(s) \frac{x^{\frac{3}{2}+\epsilon+it}}{\frac{3}{2}+\epsilon+it}ds \\
    &= x^{\frac{3}{2}+\epsilon}\left( \int_{0\leq |t|\leq 1} + \int_{1 \leq |t| \leq T} \right)F_j^{(2)}\left(\frac{3}{2}+\epsilon+it\right) \frac{x^{it}}{\frac{3}{2}+\epsilon+it}idt\\
    &= I_1+I_2.
\end{align*}

Now \begin{align*}
    I_2 &\ll x^{\frac{3}{2}+\epsilon} \int_{1}^{T}|L(\frac{1}{2}+\epsilon+it, \mathrm{sym}^j f)| \frac{1}{t} dt \quad ( \text{ by Lemma } \ref{lem:F2})\\
    &\ll x^{\frac{3}{2}+\epsilon} \log T\text{ }\sup_{1 \leq T_1 \leq T}\frac{1}{T_1}\int_{T_1}^{2T_1}|L(\frac{1}{2}+\epsilon+it, \mathrm{sym}^j f)| dt ( \text{ by Lemma } \ref{lem:int}  \\
    & \ll x^{\frac{3}{2}+\epsilon}\log T\text{ }\sup_{1 \leq T_1 \leq T}\frac{1}{T_1} \left( \int_{T_1}^{2T_1}|L(\frac{1}{2}+\epsilon+it, \mathrm{sym}^j f)|^2 dt \right)^{\frac{1}{2}}\left( \int_{T_1}^{2T_1} 1 dt \right)^{\frac{1}{2}} \\
    & \ll x^{\frac{3}{2}+\epsilon} \sup_{1 \leq T_1 \leq T}\frac{1}{T_1}(T_1^{\max\{(j+1)(1-\frac{1}{2}-\epsilon), 0\} + \epsilon})^{\frac{1}{2}}T_1^\frac{1}{2} \quad \text{(using \ref{lem:genLbound})} \\
    & \ll x^{\frac{3}{2}+\epsilon}T^{\frac{j+1}{4}-\frac{1}{2}+\epsilon}.
\end{align*}

The first integral gives \begin{align*}
    I_1 &= x^{\frac{3}{2}+\epsilon}\int_{0 \leq |t| \leq 1}F_j^{(2)}\left(\frac{3}{2}+\epsilon+it\right) \frac{x^{it}}{\frac{3}{2}+\epsilon+it}dt
    \end{align*}
    The above integration is finite. If not, then \eqref{eq:sum_l2} would be infinite. As the other integral is finite, this is a contradiction. So,
    \begin{align*}
        I_1 \ll x^{\frac{3}{2}+\epsilon}.
    \end{align*}
Combining $I_1$ and $I_2$, we have \begin{align}
    J_1 \ll x^{\frac{3}{2}+\epsilon} + x^{\frac{3}{2}+\epsilon} T^{\frac{j+1}{4}-\frac{1}{2}+\epsilon}.
\end{align}

Thus we have \begin{align}
    J_1+J_2+J_3 \ll \frac{x^{2+\epsilon}}{T}+x^{\frac{3}{2}+\epsilon} T^{\frac{j+1}{4}-\frac{1}{2}+\epsilon}.
\end{align}

Now put $T = x^{\frac{2}{j+3}}$, then using \eqref{eq:r_4&l_2} we have \begin{align}
    \sum_{\substack{n\leq x}}\lambda_{\mathrm{sym}^jf}(n)r_4(n) &= O\left( x^{2-\frac{2}{j+3}+\epsilon} \right).
\end{align}

For the sum of $6$ squares, proceeding similarly, we have \begin{align}
    \sum_{n\leq x}\lambda_{\mathrm{sym}^{j}f }(n)l_3(n) = \frac{x^{3+\epsilon}}{T}+x^{\frac{5}{2}+\epsilon} T^{\frac{j+1}{4}-\frac{1}{2}+\epsilon},
\end{align}
and also 
\begin{align}
    \sum_{n\leq x}\lambda_{\mathrm{sym}^{j}f }(n)v_3(n) = \frac{x^{3+\epsilon}}{T}+x^{\frac{5}{2}+\epsilon} T^{\frac{j+1}{4}-\frac{1}{2}+\epsilon},
\end{align}
In both equation, if we put $T = x^{\frac{2}{j+3}}$, then using \eqref{eq:r_6&l_3&v_3}, we have \begin{align}
    \sum_{\substack{n\leq x}}\lambda_{\mathrm{sym}^jf}(n)r_6(n) &= O\left( x^{3-\frac{2}{j+3}+\epsilon} \right).
\end{align}

In the case of the sum of $8$ squares, we will again get the same result as in previous cases. Thus, we will have in general \begin{align}
    \sum_{\substack{n\leq x}}\lambda_{\mathrm{sym}^jf}(n)r_m(n) &= O\left( x^{\frac{m}{2}-\frac{2}{j+3}+\epsilon} \right),
\end{align}
for $m=4,6,8$.

\section{Proof of Theorem \ref{thm:main1(0)}} \label{sec:main1(0)}

We first calculate $\sum_{n\leq x}\lambda_{\mathrm{sym}^{j}f}^2(n)$. We begin by applying Perron's formula to $F_{j}^{*(0)}(s)$ with $\eta=1+\epsilon$, and $10\leq T\leq x$. Thus we have,
\begin{align*}
\sum_{n\leq x}\lambda_{\mathrm{sym}^{j}f}^2(n)
&= \frac{1}{2\pi i}\int_{\eta-iT}^{\eta+iT}F_{j}^{*(0)}(s)\frac{x^{s}}{s}ds + O\left(\frac{x^{1+\epsilon}}{T}\right).
\end{align*}

We move the line of integration to $\Re(s)=\frac{1}{2}+\epsilon$, and by Cauchy's residue theorem, there is only one simple pole at $s=1$ due to the factor $\zeta(s)$, we get from $F^{*(0)}_{j}(s)$ in the Lemma \ref{lem:F0*}. Therefore, we have

\begin{align}
    \sum_{n\leq x}\lambda_{\mathrm{sym}^{j}f}^2(n)l_2(n) & = a_{j,f,0}(1)x+\frac{1}{2\pi i}\left\{\int_{\frac{1}{2}+\epsilon-iT}^{\frac{1}{2}+\epsilon+iT}+\int_{1+\epsilon-iT}^{\frac{1}{2}+\epsilon-iT}+\int_{\frac{1}{2}+\epsilon+iT}^{1+\epsilon+iT}\right\}F_{j}^{*(0)}(s)\frac{x^{s}}{s}ds \\
    &\quad + O\left(\frac{x^{1+\epsilon}}{T}\right) \\
    &= a_{j,f,0}(1)x^+\frac{1}{2\pi i}(J_{1}+J_{2}+J_{3})+O\left(\frac{x^{2+\epsilon}}{T}\right) \quad \text{(say)} ,
\end{align}

where \begin{align} \label{ajf0}
    a_{j,f,0}(1) &= \frac{1}{2}\prod_{n=1}^{j}L(1,\mathrm{sym}^{2n}f)H_j^{(2)}(1).
\end{align}

Now using the lemmas \ref{lem:F0*}, \ref{lem:zetabound}, \ref{lem:LSymFChi}, \ref{lem:genLbound}, we have

\begin{align}
    J_2+J_3 &\ll \left| \left( \int_{1+\epsilon-iT}^{\frac{1}{2}+\epsilon-iT}+\int_{\frac{1}{2}+\epsilon+iT}^{1+\epsilon+iT} \right)F_j^{*(0)}(s)\frac{x^s}{s}\right| \\
    &= \int_{\frac{1}{2}+\epsilon}^{1+\epsilon} \left|\zeta(\sigma+iT)\prod_{n=1}^{j}L(\sigma +iT, \mathrm{sym}^{2n} f)\right|\frac{x^{\sigma}}{T}d\sigma \\
    &\ll \frac{1}{T}\max_{\frac{1}{2}+\epsilon \leq \sigma \leq 1+\epsilon} x^\sigma T^{(\frac{13}{42}+\frac{6}{5}+\sum_{2\leq n\leq j}\frac{2n+1}{2})(1-\sigma)+\epsilon}\\
    &\ll \frac{1}{T}\max_{\frac{1}{2}+\epsilon \leq \sigma \leq 1+\epsilon}x^\sigma T^{(\frac{(j+1)^2}{2}-\frac{103}{210})(1-\sigma)+\epsilon}.
\end{align}

The above function involving $\sigma$ is monotonic, so the maximum happens at the extreme points. We treat both boundary points as upper bounds.

\begin{align*}
    J_2+J_3 &\ll \frac{1}{T}\left[ x^{(1+\epsilon)} T^{(\frac{(j+1)^2}{2}-\frac{103}{210})(\epsilon)+\epsilon} +x^{(\frac{1}{2}+\epsilon)} T^{(\frac{(j+1)^2}{2}-\frac{103}{210})(\frac{1}{2}-\epsilon)+\epsilon}\right] \\
    &\ll \frac{x^{1+\epsilon}}{T} + x^{\frac{1}{2}+\epsilon}T^{\frac{(j+1)^2}{4}-\frac{523}{420}+\epsilon} .
\end{align*}

Contribution of the left vertical line integral ($J_{1}$) in absolute value (using Lemmas \ref{lem:F0*}, \ref{lem:LSigmaChi}, \ref{lem:zetabound}, \ref{lem:LSymFChi}, \ref{lem:genLbound} and H\"older's inequality) is
\begin{align*}
J_1 &\ll \int_{\frac{1}{2}+\epsilon-iT}^{\frac{1}{2}+\epsilon+iT}\left|\zeta(\tfrac{1}{2}+\epsilon+it)\prod_{n=1}^{j}L(\tfrac{1}{2}+\epsilon+it,\mathrm{sym}^{2n}f)\right|\frac{x^{\frac{1}{2}+\epsilon}}{|t|}dt \\
&\ll x^{\frac{1}{2}+\epsilon} + x^{\frac{1}{2}+\epsilon}\int_{1 \leq |t| \leq T}\left|\zeta(\tfrac{1}{2}+\epsilon+it)\prod_{n=1}^{j}L(\tfrac{1}{2}+\epsilon+it,\mathrm{sym}^{2n}f)\right|\frac{1}{|t|}dt \\
&\ll  x^{\frac{1}{2}+\epsilon} + x^{\frac{1}{2}+\epsilon} \log T \text{ } \sup_{1 \leq T_1 \leq T}\frac{1}{T_1} \int_{T_1}^{2T_1}\left|\zeta(\tfrac{1}{2}+\epsilon+it)\prod_{n=1}^{j}L(\tfrac{1}{2}+\epsilon+it,\mathrm{sym}^{2n}f)\right| dt \\
&= x^{\frac{1}{2}+\epsilon} + x^{\frac{1}{2}+\epsilon}I_2,
\end{align*}

where the bounds of $I_2$ are given as follows.

\begin{align}
    I_2 &=\log T \sup_{1 \leq T_1 \leq T} \frac{1}{T_1} \int_{T_1}^{2T_1}\left|\zeta(\tfrac{1}{2}+\epsilon+it)\prod_{n=1}^{j}L(\tfrac{1}{2}+\epsilon+it,\mathrm{sym}^{2n}f)\right| dt \\
    &\ll T^\epsilon \sup_{1 \leq T_1 \leq T}\frac{1}{T_1}\left(  \int_{T_1}^{2T_1}|\zeta(\frac{1}{2}+\epsilon+it)|^{4}dt \right)^\frac{1}{4}\left(\int_{T_1}^{2T_1}|L(\frac{1}{2}+\epsilon+it, \mathrm{sym}^{2}f)|^4dt\right)^\frac{1}{4} \\
    & \times \left( \int_{T_1}^{2T_1}\left| \prod_{n=2}^{j}L(\frac{1}{2}+\epsilon+it, \mathrm{sym}^{2n} f) \right|^{2}dt\right)^{\frac{1}{2}}\\
    &\ll T^{\frac{1}{4}+\epsilon + 6(\frac{1}{2}+\epsilon)\frac{1}{4}+ ((j+1)^2-4)(\frac{1}{2}-\epsilon)\frac{1}{2}-1} = T^{\frac{(j+1)^2}{4}-1+\epsilon}.
\end{align}

Thus, we have \begin{align}
    J_1+J_2+J_3 \ll \frac{x^{1+\epsilon}}{T} + x^{\frac{1}{2}+\epsilon}T^{\frac{(j+1)^2}{4}-1+\epsilon}.
\end{align}

Thus \begin{align}
    \sum_{n\leq x}\lambda_{\mathrm{sym}^{j}f}^2(n) = a_{j,f,0}(1)x + O \left( \frac{x^{1+\epsilon}}{T}+x^{\frac{1}{2}+\epsilon} T^{\frac{(j+1)^2}{4}-1+\epsilon} \right).
\end{align}

Now put $T=x^{\frac{2}{(j+1)^2}}$, then we have \begin{align}
    \sum_{\substack{n\leq x}}\lambda_{\mathrm{sym}^jf}^2(n) &= c_{0,j,f}(1)x + O \left(x^{1-\frac{2}{(j+1)^2}+\epsilon}\right),
\end{align}

where $a_{j,f,0}(1)$ is given by \eqref{ajf0}.

\section{Proof of Theorem \ref{thm:main1}} \label{sec:main1}

We consider the sum $\sum_{n\leq x}\lambda_{\mathrm{sym}^{j}f}^2(n)l_1(n)$. We begin by applying Perron's formula to $F_{j}^{*(1)}(s)$ with $\eta=1+\epsilon$, and $10\leq T\leq x$. Thus we have,
\begin{align*}
\sum_{n\leq x}\lambda_{\mathrm{sym}^{j}f}^2(n)l_1(n)
&= \frac{1}{2\pi i}\int_{\eta-iT}^{\eta+iT}F_{j}^{*(1)}(s)\frac{x^{s}}{s}ds + O\left(\frac{x^{1+\epsilon}}{T}\right).
\end{align*}

We move the line of integration to $\Re(s)=\frac{1}{2}+\epsilon$, and by Cauchy's residue theorem, there is only one simple pole at $s=1$ due to the factor $\zeta(s)$, we get from $F^{*(1)}_{j}(s)$ in the Lemma \ref{lem:F1*}.

This contributes a residue, which is $c_{j,f}(1)x$, where 
\begin{align} \label{eq:cjf}
c_{j,f}(1) &= L(1,\chi_4)\prod_{n=1}^{j}L(1,\mathrm{sym}^{2n}f)L(1,\mathrm{sym}^{2n}f\otimes\chi_4)H_{j}^{*(1)}(1).
\end{align}

So, we obtain
\begin{align*}
\sum_{n\leq x}\lambda_{\mathrm{sym}^{j}f}^2(n)l_1(n) &= c_{j,f}(1)x + \frac{1}{2\pi i}\left\{\int_{\frac{1}{2}+\epsilon-iT}^{\frac{1}{2}+\epsilon+iT}+\int_{1+\epsilon-iT}^{\frac{1}{2}+\epsilon-iT}+\int_{\frac{1}{2}+\epsilon+iT}^{1+\epsilon+iT}\right\}F_{j}^{*(1)}(s)\frac{x^{s}}{s}ds \\
&\quad + O\left(\frac{x^{1+\epsilon}}{T}\right) \\
&= c_{j,f}(1)x + \frac{1}{2\pi i}(J_{1}+J_{2}+J_{3})+O\left(\frac{x^{1+\epsilon}}{T}\right), \quad \text{(say)}.
\end{align*}

Contribution of horizontal line integrals ($J_{2}$ and $J_{3}$) in absolute value (using Lemmas \ref{lem:F1*}, \ref{lem:zetabound}, \ref{lem:LSymFChi}, \ref{lem:LSigmaChi} and \ref{lem:genLbound}) is

\begin{align*}
    J_2+J_3 &\ll \left| \left( \int_{1+\epsilon-iT}^{\frac{1}{2}+\epsilon-iT}+\int_{\frac{1}{2}+\epsilon+iT}^{1+\epsilon+iT} \right)F_j^{*(1)}(s)\frac{x^s}{s} ds \right| \\
    &\ll  \int_{\frac{1}{2}+\epsilon}^{1+\epsilon} \left|\zeta(\sigma+iT)L(\sigma+iT, \chi_4)\prod_{n=1}^{j}L(\sigma+iT, \mathrm{sym}^{2n} f)L(\sigma+iT, \mathrm{sym}^{2n} f \otimes \chi_4 )\right|\frac{1}{T}x^\sigma d\sigma \\
    &\ll \frac{1}{T}\max_{\frac{1}{2}+\epsilon \leq \sigma \leq 1+\epsilon} x^\sigma T^{(\frac{13}{42}+\frac{1}{3}+\frac{6}{5}+\frac{67}{46}+2\sum_{2\leq n\leq j}\frac{2n+1}{2})(1-\sigma)+\epsilon}\\
    &\ll \frac{1}{T}\max_{\frac{1}{2}+\epsilon \leq \sigma \leq 1+\epsilon}x^\sigma T^{\left((j+1)^2-\frac{564}{805} \right )(1-\sigma)+\epsilon}.
\end{align*}

The above function involving $\sigma$ is monotonic, so the maximum happens at the end points. We treat both boundary points as upper bounds. Thus, we have 
\begin{align*}
    J_2+J_3 &\ll \frac{x^{1+\epsilon}}{T} + x^{\frac{1}{2}+\epsilon}T^{\frac{(j+1)^2}{2}-\frac{1087}{805}+\epsilon} .
\end{align*}

Contribution of the left vertical line integral ($J_{1}$) in absolute value (using Lemmas \ref{lem:F1*}, \ref{lem:LSigmaChi}, \ref{eq:intLSigmafChi} \ref{lem:zetabound}, \ref{lem:genLbound} and H\"older's inequality) is

\begin{align*}
J_1 &\ll \int_{\frac{1}{2}+\epsilon-iT}^{\frac{1}{2}+\epsilon+iT}\left|\zeta(\tfrac{1}{2}+\epsilon+it)L(\tfrac{1}{2}+\epsilon+it, \chi_4)\right| \\
& \times \left|\prod_{n=1}^{j}L(\tfrac{1}{2}+\epsilon+it,\mathrm{sym}^{2n}f )L(\tfrac{1}{2}+\epsilon+it,\mathrm{sym}^{2n}f \otimes \chi_4 )\right|\frac{x^{\frac{1}{2}+\epsilon}}{|t|}dt \\
&\ll x^{\frac{1}{2}+\epsilon} + x^{\frac{1}{2}+\epsilon}\int_{1 \leq |t| \leq T}\left|\zeta(\tfrac{1}{2}+\epsilon+it)L(\tfrac{1}{2}+\epsilon+it, \chi_4) \right| \\ 
& \times \left|\prod_{n=1}^{j}L(\tfrac{1}{2}+\epsilon+it,\mathrm{sym}^{2n}f)L(\tfrac{1}{2}+\epsilon+it,\mathrm{sym}^{2n}f \otimes \chi_4 )\right|\frac{1}{|t|}dt \\
&\ll  x^{\frac{1}{2}+\epsilon} + x^{\frac{1}{2}+\epsilon} \log T \text{ } \sup_{1 \leq T_1 \leq T}\frac{1}{T_1} \int_{T_1}^{2T_1}\left|\zeta(\tfrac{1}{2}+\epsilon+it)L(\tfrac{1}{2}+\epsilon+it, \chi_4) \right| \\ 
& \times \left|\prod_{n=1}^{j}L(\tfrac{1}{2}+\epsilon+it,\mathrm{sym}^{2n}f)L(\tfrac{1}{2}+\epsilon+it,\mathrm{sym}^{2n}f \otimes \chi_4 )\right| dt \\
&= x^{\frac{1}{2}+\epsilon} + x^{\frac{1}{2}+\epsilon}I_2,
\end{align*}
where an upper bounds for $I_2$ is given as follows.

\begin{align}
    I_2 &=\log T \text{ } \sup_{1 \leq T_1 \leq T}\frac{1}{T_1} \int_{T_1}^{2T_1}\left|\zeta(\tfrac{1}{2}+\epsilon+it)L(\tfrac{1}{2}+\epsilon+it, \chi_4) \right| \\ 
    & \times \left|\prod_{n=1}^{j}L(\tfrac{1}{2}+\epsilon+it,\mathrm{sym}^{2n}f)L(\tfrac{1}{2}+\epsilon+it,\mathrm{sym}^{2n}f \otimes \chi_4 )\right| dt \\
    &\ll T^\epsilon \sup_{1 \leq T_1 \leq T}\frac{1}{T_1}\left(  \int_{T_1}^{2T_1}|\zeta(\frac{1}{2}+\epsilon+it)|^{4}dt \right)^\frac{1}{4}\left(\int_{T_1}^{2T_1}|L(\tfrac{1}{2}+\epsilon+it, \chi_4)|^4dt\right)^\frac{1}{4} \\
    & \times \left( \int_{T_1}^{2T_1}\left| \prod_{n=1}^{j}L(\tfrac{1}{2}+\epsilon+it,\mathrm{sym}^{2n}f)L(\tfrac{1}{2}+\epsilon+it,\mathrm{sym}^{2n}f \otimes \chi_4 ) \right|^{2}dt\right)^{\frac{1}{2}}\\
    &\ll \frac{1}{T_1}T_1^{(\frac{1}{4}+\epsilon) + (\frac{1}{4}+\epsilon)+ 2 \sum_{1 \leq n \leq j}(2n+1)(\frac{1}{2}-\epsilon)\frac{1}{2}}\\
    &=  T^{\frac{(j+1)^2}{2}-1+\epsilon}.
\end{align}

So we have \begin{align}
    J_1 \ll x^{\frac{1}{2}+\epsilon} + x^{\frac{1}{2}+\epsilon}T^{\frac{(j+1)^2}{2}-1+\epsilon}
\end{align}

and \begin{align}
    J_1+J_2+J_3 \ll \frac{x^{1+\epsilon}}{T}+ x^{\frac{1}{2}+\epsilon} T^{\frac{(j+1)^2}{2}-1+\epsilon}.
\end{align}

Thus \begin{align}
    \sum_{n\leq x}\lambda_{\mathrm{sym}^{j}f}^2(n)l_1(n) = c_{j,f}(1)x + O \left( \frac{x^{1+\epsilon}}{T}+ x^{\frac{1}{2}+\epsilon} T^{\frac{(j+1)^2}{2}-1+\epsilon} \right).
\end{align}

After we put $T = x^{\frac{1}{(j+1)^2}}$, we have \begin{align}
    \sum_{n\leq x}\lambda_{\mathrm{sym}^{j}f}^2(n)l_1(n) = c_{j,f}(1)x + O \left( x^{1-\frac{1}{(j+1)^2}+\epsilon} \right).
\end{align}

Now from \eqref{eq:r_2&l_1}, we have our result.

\section{Proof of Theorem \ref{thm:main1(468)}} \label{sec:main1(468)}

We first calculate $\sum_{n\leq x}\lambda_{\mathrm{sym}^{j}f}^2(n)l_2(n)$. We begin by applying Perron's formula to $F_{j}^{*(2)}(s)$ with $\eta=2+\epsilon$, and $10\leq T\leq x$. Thus we have,
\begin{align*}
\sum_{n\leq x}\lambda_{\mathrm{sym}^{j}f}^2(n)l_2(n)
&= \frac{1}{2\pi i}\int_{\eta-iT}^{\eta+iT}F_{j}^{*(2)}(s)\frac{x^{s}}{s}ds + O\left(\frac{x^{2+\epsilon}}{T}\right).
\end{align*}

We move the line of integration to $\Re(s)=\frac{3}{2}+\epsilon$, and by Cauchy's residue theorem, there is only one simple pole at $s=2$ due to the factor $\zeta(s)$, we get from $F^{*2}_{j}(s)$ in the Lemma \ref{lem:F2*}. We have,

\begin{align}
    \sum_{n\leq x}\lambda_{\mathrm{sym}^{j}f}^2(n)l_2(n) & = a_{j,f,2}(2)x^2+\frac{1}{2\pi i}\left\{\int_{\frac{3}{2}+\epsilon-iT}^{\frac{3}{2}+\epsilon+iT}+\int_{2+\epsilon-iT}^{\frac{3}{2}+\epsilon-iT}+\int_{\frac{3}{2}+\epsilon+iT}^{2+\epsilon+iT}\right\}F_{j}^{*(2)}(s)\frac{x^{s}}{s}ds \\
    &\quad + O\left(\frac{x^{2+\epsilon}}{T}\right) \\
    &= a_{j,f,2}(2)x^2+\frac{1}{2\pi i}(J_{1}+J_{2}+J_{3})+O\left(\frac{x^{2+\epsilon}}{T}\right) \quad \text{(say)} ,
\end{align}

where \begin{align} \label{ajf2}
    a_{j,f,2}(2) &= \frac{1}{2}\zeta(2)\prod_{n=1}^{j}L(2,\mathrm{sym}^{2n}f)L(1,\mathrm{sym}^{2n}f)H_j^2(2).
\end{align}

Now using the lemmas \ref{lem:F2*}, \ref{lem:zetabound}, \ref{lem:LSymFChi}, \ref{lem:genLbound}, we have

\begin{align}
    J_2+J_3 &\ll \left| \left( \int_{2+\epsilon-iT}^{\frac{3}{2}+\epsilon-iT}+\int_{\frac{3}{2}+\epsilon+iT}^{2+\epsilon+iT} \right)F_j^{*(2)}(s)\frac{x^s}{s}\right| \\
    &\ll  \int_{\frac{3}{2}+\epsilon}^{2+\epsilon} \left|\zeta(\sigma-1+iT)\prod_{n=1}^{j}L(\sigma+iT-1, \mathrm{sym}^{2n} f)\right|\frac{1}{T}x^\sigma d\sigma \\
    &= \int_{\frac{1}{2}+\epsilon}^{1+\epsilon} \left|\zeta(\sigma+iT)\prod_{n=1}^{j}L(\sigma +iT, \mathrm{sym}^{2n} f)\right|\frac{x^{\sigma+1}}{T}d\sigma \\
    &\ll \frac{x}{T}\max_{\frac{1}{2}+\epsilon \leq \sigma \leq 1+\epsilon} x^\sigma T^{(\frac{13}{42}+\frac{6}{5}+\sum_{2\leq n\leq j}\frac{2n+1}{2})(1-\sigma)+\epsilon}\\
    &\ll \frac{x}{T}\max_{\frac{1}{2}+\epsilon \leq \sigma \leq 1+\epsilon}x^\sigma T^{(\frac{(j+1)^2}{2}-\frac{103}{210})(1-\sigma)+\epsilon}.
\end{align}

The above function involving $\sigma$ is monotonic, so the maximum happens at the extreme points. We treat both boundary points as upper bounds.

\begin{align*}
    J_2+J_3 &\ll \frac{x}{T}\left[ x^{(1+\epsilon)} T^{(\frac{(j+1)^2}{2}-\frac{103}{210})(\epsilon)+\epsilon} +x^{(\frac{1}{2}+\epsilon)} T^{(\frac{(j+1)^2}{2}-\frac{103}{210})(\frac{1}{2}-\epsilon)+\epsilon}\right] \\
    &\ll \frac{x^{2+\epsilon}}{T} + x^{\frac{3}{2}+\epsilon}T^{\frac{(j+1)^2}{4}-\frac{523}{420}+\epsilon} .
\end{align*}

Contribution of the left vertical line integral ($J_{1}$) in absolute value (using Lemmas \ref{lem:F2*}, \ref{lem:LSigmaChi}, \ref{lem:zetabound}, \ref{lem:LSymFChi}, \ref{lem:genLbound} and H\"older's inequality) is
\begin{align*}
J_1 &\ll \int_{\frac{3}{2}+\epsilon-iT}^{\frac{3}{2}+\epsilon+iT}\left|\zeta(\tfrac{1}{2}+\epsilon+it)\prod_{n=1}^{j}L(\tfrac{1}{2}+\epsilon+it,\mathrm{sym}^{2n}f)\right|\frac{x^{\frac{3}{2}+\epsilon}}{|t|}dt \\
&\ll x^{\frac{3}{2}+\epsilon} + x^{\frac{3}{2}+\epsilon}\int_{1 \leq |t| \leq T}\left|\zeta(\tfrac{1}{2}+\epsilon+it)\prod_{n=1}^{j}L(\tfrac{1}{2}+\epsilon+it,\mathrm{sym}^{2n}f)\right|\frac{1}{|t|}dt \\
&\ll  x^{\frac{3}{2}+\epsilon} + x^{\frac{3}{2}+\epsilon} \log T \text{ } \sup_{1 \leq T_1 \leq T}\frac{1}{T_1} \int_{T_1}^{2T_1}\left|\zeta(\tfrac{1}{2}+\epsilon+it)\prod_{n=1}^{j}L(\tfrac{1}{2}+\epsilon+it,\mathrm{sym}^{2n}f)\right| dt \\
&= x^{\frac{3}{2}+\epsilon} + x^{\frac{3}{2}+\epsilon}I_2,
\end{align*}

where the bounds of $I_2$ is given by as follows.

\begin{align}
    I_2 &=\log T \sup_{1 \leq T_1 \leq T} \frac{1}{T_1} \int_{T_1}^{2T_1}\left|\zeta(\tfrac{1}{2}+\epsilon+it)\prod_{n=1}^{j}L(\tfrac{1}{2}+\epsilon+it,\mathrm{sym}^{2n}f)\right| dt \\
    &\ll T^\epsilon \sup_{1 \leq T_1 \leq T}\frac{1}{T_1}\left(  \int_{T_1}^{2T_1}|\zeta(\frac{1}{2}+\epsilon+it)|^{4}dt \right)^\frac{1}{4}\left(\int_{T_1}^{2T_1}|L(\frac{1}{2}+\epsilon+it, \mathrm{sym}^{2}f)|^4dt\right)^\frac{1}{4} \\
    & \times \left( \int_{T_1}^{2T_1}\left| \prod_{n=2}^{j}L(\frac{1}{2}+\epsilon+it, \mathrm{sym}^{2n} f) \right|^{2}dt\right)^{\frac{1}{2}}\\
    &\ll T^{\frac{1}{4}+\epsilon + 6(\frac{1}{2}+\epsilon)\frac{1}{4}+ ((j+1)^2-4)(\frac{1}{2}-\epsilon)\frac{1}{2}-1} = T^{\frac{(j+1)^2}{4}-1+\epsilon}.
\end{align}

Thus, we have \begin{align}
    J_1+J_2+J_3 \ll \frac{x^{2+\epsilon}}{T} + x^{\frac{3}{2}+\epsilon}T^{\frac{(j+1)^2}{4}-1+\epsilon}.
\end{align}

Thus \begin{align}
    \sum_{n\leq x}\lambda_{\mathrm{sym}^{j}f}^2(n)l_2(n) = a_{j,f,2}(2)x^2 + O \left( \frac{x^{2+\epsilon}}{T}+x^{\frac{3}{2}+\epsilon} T^{\frac{(j+1)^2}{4}-1+\epsilon} \right).
\end{align}

Now put $T=x^{\frac{2}{(j+1)^2}}$, then we have \begin{align}
    \sum_{\substack{n\leq x}}\lambda_{\mathrm{sym}^jf}^2(n)l_2(n) &= c_{j,f,2}(2)x^{2} + O \left(x^{2-\frac{2}{(j+1)^2}+\epsilon}\right),
\end{align}

where $c_{j,f,2}(2)$ is given by \eqref{ajf2}.

Now we will consider the sum of $6$ squares, and we will calculate the parts $\sum_{\substack{n\leq x}}\lambda_{\mathrm{sym}^jf}^2(n)l_3(n)$ and $\sum_{\substack{n\leq x}}\lambda_{\mathrm{sym}^jf}^2(n)v_3(n)$. Using Perron's formula to $F_{j_1}^{*(3)}(s)$ with $\eta = 3+\epsilon$ and after moving the line of integration to $\eta = \frac{5}{2}+\epsilon$, we have  
\begin{align}
    \sum_{n\leq x}\lambda_{\mathrm{sym}^{j}f}^2(n)l_3(n) & = a_{j,f,3}(3)x^3+\frac{1}{2\pi i}\left\{\int_{\frac{5}{2}+\epsilon-iT}^{\frac{5}{2}+\epsilon+iT}+\int_{3+\epsilon-iT}^{\frac{5}{2}+\epsilon-iT}+\int_{\frac{5}{2}+\epsilon+iT}^{3+\epsilon+iT}\right\}F_{j_1}^{*(3)}(s)\frac{x^{s}}{s}ds \\
    &\quad + O\left(\frac{x^{3+\epsilon}}{T}\right) \\
    &= a_{j,f,3}(3)x^2+\frac{1}{2\pi i}(J_{1}+J_{2}+J_{3})+O\left(\frac{x^{3+\epsilon}}{T}\right) \quad \text{(say)} ,
\end{align}

where \begin{align} \label{ajf3}
    a_{j_1,f,3}(3) &= \frac{1}{3}L(3,\chi_4)\prod_{n=1}^{j}L(3,\mathrm{sym}^{2n}f \otimes \chi_4)L(1,\mathrm{sym}^{2n}f)H_j^3(3).
\end{align}

Now, using Lemmas\ref{lem:zetabound}, \ref{lem:LSymFChi}, \ref{lem:genLbound} and proceeding as before, we have 
\begin{align} \label{eq:convSuml3}
    \sum_{\substack{n\leq x}}\lambda_{\mathrm{sym}^jf}^2(n)l_3(n) &= a_{j_1,f,3}(3)x^3 + O \left( \frac{x^{3+\epsilon}}{T}+ {x^{\frac{5}{2}+\epsilon}T^{\frac{(j+1)^2}{4}-1+\epsilon}} \right).
\end{align}

Again using Perron's formula to $F_{j_2}^{*(3)}(s)$ with $\eta = 3+\epsilon$ and after moving the line of integration to $\eta = \frac{5}{2}+\epsilon$, we have  
\begin{align}
    \sum_{n\leq x}\lambda_{\mathrm{sym}^{j}f}^2(n)v_3(n) & = \frac{1}{2\pi i}\left\{\int_{\frac{5}{2}+\epsilon-iT}^{\frac{5}{2}+\epsilon+iT}+\int_{3+\epsilon-iT}^{\frac{5}{2}+\epsilon-iT}+\int_{\frac{5}{2}+\epsilon+iT}^{3+\epsilon+iT}\right\}F_{j_2}^{*(3)}(s)\frac{x^{s}}{s}ds \\
    &\quad + O\left(\frac{x^{3+\epsilon}}{T}\right) \\
    &= \frac{1}{2\pi i}(J_{1}+J_{2}+J_{3})+O\left(\frac{x^{3+\epsilon}}{T}\right) \quad \text{(say)}.
\end{align}

Now using Lemmas\ref{lem:LSigmaChi}, \ref{lem:LSymFChi}, \ref{lem:genLbound} and proceeding as before, we have 
\begin{align} \label{eq:convSumv3}
    \sum_{\substack{n\leq x}}\lambda_{\mathrm{sym}^jf}^2(n)v_3(n) &= O \left( \frac{x^{3+\epsilon}}{T}+ {x^{\frac{5}{2}+\epsilon}T^{\frac{(j+1)^2}{4}-1+\epsilon}} \right).
\end{align}

Combining \eqref{eq:convSuml3}, \eqref{eq:convSumv3} and \eqref{eq:r_6&l_3&v_3}, we have \begin{align} 
    \sum_{\substack{n\leq x}}\lambda_{\mathrm{sym}^jf}^2(n)r_6(n) &= 16a_{j_1,f,3}(3)x^3 + O \left( \frac{x^{3+\epsilon}}{T}+ {x^{\frac{5}{2}+\epsilon}T^{\frac{(j+1)^2}{4}-1+\epsilon}} \right).
\end{align}

Now,  if we put $T = x^{\frac{2}{(j+1)^2}}$, then we have \begin{align}
    \sum_{n\leq x}\lambda_{\mathrm{sym}^{j}f}^2(n)r_{6}(n) = 16a_{j_1,f,3}(3)x^3 + O\left( x^{3+\epsilon-\frac{2}{(j+1)^2}} \right).
\end{align}

In the case of a sum of $8$ squares, we also proceed as before, and we have 
\begin{align}
    \sum_{\substack{n\leq x }}\lambda_{\mathrm{sym}^jf}^2(n)r_8(n) &= 16a_{j,f,4}(4)x^{4} + O \left( x^{4-\frac{2}{(j+1)^2}+\epsilon}\right).
\end{align}

Combinig everything, for $m=4,6,8$, we have 
\begin{align}
    \sum_{\substack{n\leq x }}\lambda_{\mathrm{sym}^jf}^2(n)r_m(n) &= C_{j,f,m}\left(\frac{m}{2}\right)x^{\frac{m}{2}} + O \left( x^{\frac{m}{2}-\frac{2}{(j+1)^2}+\epsilon} \right).
\end{align}

\section{Proof of Theorem \ref{thm:main(10,12)}} \label{sec:main(10,12)}

For $m=10$, we first consider the terms $\sum_{\substack{n\leq x}}\lambda_{\mathrm{sym}^jf}(n)l_5(n)$ and $\sum_{\substack{n\leq x}}\lambda_{\mathrm{sym}^jf}(n)v_5(n)$. Proceeding as in the case of the sum of $6$ squares, we will have \begin{align}
    \sum_{\substack{n\leq x}}\lambda_{\mathrm{sym}^jf}^2(n)l_5(n) &= a_{j_1,f,5}(5)x^5 + O\left( x^{5+\epsilon-\frac{2}{(j+1)^2}} \right),
\end{align}

where $a_{j_1,f,5}(5) = \frac{1}{5}L(5,\chi_4)\prod_{n=1}^{j}L(1,\mathrm{sym}^{2n} f)L(5,\mathrm{sym}^{2n}f\otimes\chi_4)H_{j}^{(5)}(5)$ and 

\begin{align}
    \sum_{\substack{n\leq x}}\lambda_{\mathrm{sym}^jf}^2(n)v_5(n) &= O\left( x^{5+\epsilon-\frac{2}{(j+1)^2}} \right).
\end{align}
We know $a_n = O(n^3)$ (see in~\cite{Grosswald}). Let $$\overline{F}_{j}^{(5)}(s) = \sum_{n=1}^{\infty}\frac{\lambda_{\mathrm{sym}^jf}(n)a_n}{n^s}.$$

we consider the sum $\sum_{n \leq x} \lambda_{\mathrm{sym}^j f}(n)a_n$, where we know that $a_n = O(n^3).$ There exists a $G \in \mathbb{N}$ such that $a_n \ll n^3$ for all $n > G$. We have 
\begin{align*}
    \sum_{n \leq x} \lambda_{\mathrm{sym}^j f}(n)a_n &= \left( \sum_{n \leq G} + \sum_{G < n \leq x} \right) \lambda_{\mathrm{sym}^j f}(n)a_n = O(1) + \sum_{G < n \leq x}\lambda_{\mathrm{sym}^j f}(n)n^3 .
\end{align*} 

Since $\sum_{G <n \leq x} \lambda_{\mathrm{sym}^j f}(n)n^3 = \left( \sum_{n \leq x} - \sum_{n \leq G} \right)\lambda_{\mathrm{sym}^j f}(n)n^3$, we have
\begin{align*}
    |\sum_{n \leq x}\lambda_{\mathrm{sym}^j f}(n)a_n| &\ll 1 + |\sum_{n\leq x}\lambda_{\mathrm{sym}^j f}(n)n^3| +|\sum_{n\leq G}\lambda_{\mathrm{sym}^j f}(n)n^3|.
\end{align*}
Next, we will evaluate the sum $\sum_{n\leq x}\lambda_{\mathrm{sym}^j f}(n)n^3$ using Perron's formula as above. Take $\eta = 4+\epsilon$ and $10 \leq T \leq x$. Then \begin{align*}
    \sum_{n\leq x}\lambda_{\mathrm{sym}^j f}(n)n^3 &= \frac{1}{2 \pi i}\int_{\eta - iT}^{\eta +iT}L(s-3, \mathrm{sym}^j f)\frac{x^s}{s}ds + O\left(\frac{x^{4+\epsilon}}{T}\right)\\
    &= \frac{1}{2 \pi i}\left\{ \int_{\frac{7}{2}+\epsilon-iT}^{\frac{7}{2}+\epsilon+iT} + \int_{4+\epsilon -iT}^{\frac{7}{2}+\epsilon-iT} + \int_{\frac{7}{2}+\epsilon+iT}^{4+\epsilon +iT}\right\}L(s-3, \mathrm{sym}^j f)\frac{x^s}{s}ds + O\left(\frac{x^{4+\epsilon}}{T}\right)\\
    &= \frac{1}{2 \pi i} \left( J_1+J_2+J_3 \right) + O\left(\frac{x^{4+\epsilon}}{T}\right).
\end{align*} 

After evaluating the integrals as above, we have 
\begin{align}
    \sum_{n\leq x}\lambda_{\mathrm{sym}^j f}(n)n^3 &= O\left(\frac{x^{4+\epsilon}}{T}+x^{\frac{7}{2}+\epsilon} T^{\frac{j+1}{4}-\frac{1}{2}+\epsilon} \right) .
\end{align}

Now putting $T= x^\frac{2}{j+3}$, we have 
\begin{align}
    \sum_{n\leq x}\lambda_{\mathrm{sym}^j f}(n)n^3 &= O\left(x^{4-\frac{2}{j+3}+\epsilon} \right) .
\end{align}
Combining everything, we have 
\begin{align}
    \sum_{n\leq x}\lambda_{\mathrm{sym}^j f}(n)r_{10}(n) = O\left( x^{5+\epsilon - \frac{2}{j+3}} \right).
\end{align}

Again, note that $b_n = O(n^3\log\log n)$ (see~\cite{Grosswald}). Thus, proceeding as above, we have \begin{align}
    \sum_{n\leq x}\lambda_{\mathrm{sym}^j f}(n)r_{12}(n) = O\left( x^{6+\epsilon - \frac{2}{j+3}} \right).
\end{align}

\section{Proof of Theorem \ref{thm:main1(10,12)}} \label{sec:main1(10,12)}

As we have seen in the previous theorem, we only have to calculate the terms $\sum_{\substack{n\leq x}}\lambda_{\mathrm{sym}^jf}^2(n)l_5(n)$ and $\sum_{\substack{n\leq x}}\lambda_{\mathrm{sym}^jf}^2(n)v_5(n)$ for $m=5$. Proceeding as in the case of the sum of $6$ squares, we have 
\begin{align}
    \sum_{n\leq x}\lambda_{\mathrm{sym}^{j}f}^2(n)r_{10}(n) = \frac{64}{5}c_{j,f,5}(5)x^5 + O\left( x^{5+\epsilon-\frac{2}{(j+1)^2}} \right),
\end{align}
where $c_{j,f,5}(5) = \frac{1}{5}L(5,\chi_4)\prod_{n=1}^{j}L(1,\mathrm{sym}^{2n}f)L(5,\mathrm{sym}^{2n}f \otimes\chi_4)$.

Similarly, for $m=12$, we only have to calculate $\sum_{\substack{n\leq x}}\lambda_{\mathrm{sym}^jf}^2(n)l_6(n)$ and we have 
\begin{align}
    \sum_{n\leq x}\lambda_{\mathrm{sym}^{j}f}^2(n)r_{12}(n) = 8c_{j,f,6}(6)x^6 + O\left( x^{6+\epsilon-\frac{2}{(j+1)^2}} \right),
\end{align}
where $c_{j,f,6}(6) = \frac{1}{6}\zeta(6)\prod_{n=1}^{j}L(1,\mathrm{sym}^{2n}f)L(6,\mathrm{sym}^{2n}f)H_{j}^{(6)}(6)$. This proves our result.

\section{Proof of Theorem \ref{thm:mainsign}} \label{sec:mainsign}
Let $S(x) = \sum_{n \leq x} \lambda_{\mathrm{sym}^jf}(n)r_2(n)$ and  $h=x^{\delta_j}$ with $A(j)\coloneqq 1-\frac{1}{(j+1)^2}  < \delta_j < 1$. Now suppose that $\{ \lambda_{\mathrm{sym}^jf}(n) | n = a_1^2 + a_2^2, a_i \in \mathbb{Z} \}$ does not change any sign in the interval $n \in (x,x+h]$ and without loss of generality suppose that the sequence stays positive in the given interval.

Using the Theorem \ref{thm:main}, we have
\begin{align} \label{eq:uppbound}
    \sum_{x < n \leq x+h}\lambda_{\mathrm{sym}^jf}^2(n)r_2(n) &= \sum_{x < n \leq x+h}\lambda_{\mathrm{sym}^jf}(n)\lambda_{\mathrm{sym}^jf}(n)r_2(n) \\
    & \ll (x+h)^{\epsilon} \sum_{x < n \leq x+h}\lambda_{\mathrm{sym}^jf}(n)r_2(n) \quad (\text{as } \lambda_{\mathrm{sym}^jf}(n) \ll n^\epsilon \text{ for any } \epsilon > 0) \\
    &\ll x^\epsilon (|S(x+h)|+|S(x)|) \\
    & \ll x^{1-\frac{1}{j+1}+\epsilon},
\end{align}
for any $\epsilon>0$.

Now using the Theorem \ref{thm:main1}, we have 
\begin{align}
    \sum_{x < n \leq x+h}\lambda_{\mathrm{sym}^jf}^2(n)r_2(n) & = Ch + O\left( x^{1-\frac{1}{(j+1)^2}+\epsilon}\right) 
    = Cx^{\delta_j} + O\left( x^{1-\frac{1}{(j+1)^2}+\epsilon}\right).
\end{align}

Lemma \ref{dominating main term} ensures that 
\begin{align}\label{lower bound}
    \sum_{x < n \leq x+h}\lambda_{\mathrm{sym}^jf}^2(n)r_2(n)  \gg  x^{\delta_j}.
\end{align}
Combining \eqref{eq:uppbound} and \eqref{lower bound}, we obtain
$$x^{\delta_j }\ll x^{1-\frac{1}{j+1}+\epsilon}$$
as $x\to \infty$ for any $\epsilon >0.$ That is, $x^{\delta_j -A(j)-\epsilon}\ll 1$ as $x\to \infty$ for any $\epsilon >0.$ In particular, choosing $\epsilon=\frac{1}{2}(\delta_j -A(j))>0$, we obtain that $x^{\epsilon} \ll 1$ as $x \to \infty$, which is a contradiction.

This implies that there exists at least one sign change in the interval $(x, x+x^{\delta_j}]$, where $x$ is sufficiently large. Similarly, we can prove that there exists at least one sign change in $(x+x^{\delta_j}, x+2x^{\delta_j}]$ and so on.

Note that $2x = x+ x^{1-\delta_j}x^{\delta_j}$, and we have that there exists at least $x^{1-\delta_j}$ number of sign changes in the interval $(x,2x]$.

\section{Proof of Theorem \ref{thm:main1(468)sign}} \label{sec:main1(468)sign}

Let $S_m(x) = \sum_{n \leq x} \lambda_{\mathrm{sym}^jf}(n)r_m(n)$ and  $h=x^{\delta_j}$ with $1-\frac{2}{(j+1)^2}  < \delta_j < 1-\frac{1}{(j+1)^2}$. Now suppose that $\{ \lambda_{\mathrm{sym}^jf}(n) | n = \sum_{i=1}^{m}a_i^2, a_i \in \mathbb{Z} \}$ does not change any sign in the interval $n \in (x,x+h]$ and without loss of generality suppose that the sequence stays positive in the given interval.

Using Theorem \ref{thm:main(468)}, we have
\begin{align} \label{eq:uppbound1}
    \sum_{x < n \leq x+h}\lambda_{\mathrm{sym}^jf}^2(n)r_m(n) &= \sum_{x < n \leq x+h}\lambda_{\mathrm{sym}^jf}(n)\lambda_{\mathrm{sym}^jf}(n)r_m(n) \\
    & \ll (x+h)^{\epsilon} \sum_{x < n \leq x+h}\lambda_{\mathrm{sym}^jf}(n)r_m(n) \quad (\text{ as } \lambda_{\mathrm{sym}^jf}(n) \ll n^\epsilon, \text{ for all } \epsilon > 0) \\
    &\ll x^\epsilon (|S_m(x+h)|+|S_m(x)|) \\
    & \ll x^{\frac{m}{2}-\frac{2}{j+3}+\epsilon}.
\end{align}

Now using the Theorem \ref{thm:main1(468)}, we have 
\begin{align}
    \sum_{x < n \leq x+h}\lambda_{\mathrm{sym}^jf}^2(n)r_m(n) & = c_{j,f,m}(x+h)^\frac{m}{2}-c_{m,j,f}x^{\frac{m}{2}} + O\left( x^{\frac{m}{2}-\frac{2}{(j+1)^2}+\epsilon}\right) \\
    &= c_{j,f,m}(c_1x^{\frac{m}{2}-1+\delta_j}+c_2x^{\frac{m}{2}-2+2\delta_j} + \cdots +c_{\frac{m}{2}}x^{\frac{m}{2}\delta_j})+ O\left( x^{\frac{m}{2}-\frac{2}{(j+1)^2}+\epsilon}\right)\\
    &= c_{m,j,f}'x^{\frac{m}{2}-1+\delta_j} + O\left( x^{\frac{m}{2}-\frac{2}{(j+1)^2}+\epsilon}\right) \quad  \left(\text{ as } \delta_j < 1-\frac{1}{(j+1)^2} \right ),
\end{align}
where $c_{j,f,m}'$ is a constant, depending on $j,f$ and $m$.

Thus, by the Lemma \ref{dominating main term},  we have \begin{align} \label{eq:lowbound1}
    \sum_{x < n \leq x+h}\lambda_{\mathrm{sym}^jf}^2(n)r_m(n) \gg x^{\frac{m}{2}-1+\delta_j}.
\end{align}

Combining \eqref{eq:uppbound1} and \eqref{eq:lowbound1}, we have \begin{align}
    x^{\frac{m}{2}-1+\delta_j} \ll x^{\frac{m}{2}-\frac{2}{j+3}+\epsilon}.
\end{align}
This gives \begin{align}
    x^{\frac{2}{j+3}-\frac{2}{(j+1)^2}+\epsilon} \leq x^{\frac{m}{2}-1+\delta_j-\frac{m}{2}+\frac{2}{j+3}+\epsilon} \ll 1,
\end{align}
for any $\epsilon > 0$. This leads to a contradiction as $x \to +\infty$.

This implies that there exists at least one sign change in the interval $(x, x+x^{\delta_j}]$, where $x$ is sufficiently large. Similarly, we can prove that there exists at least one sign change in the interval $(x+x^{\delta_j}, x+2x^{\delta_j}]$ and so on. Thus, we have that there exists at least $x^{1-\delta_j}$ number of sign changes in the interval $(x,2x]$.

\bibliographystyle{plain}
\bibliography{references}

\end{document}